\pdfoutput=1
\documentclass{article}

\PassOptionsToPackage{numbers,sort&compress}{natbib}

  \usepackage[preprint]{neurips_2026}
    \author{Damien Scieur \\
  Samsung SAIL Montreal\\
  \texttt{damien.scieur@gmail.com}}

\usepackage[utf8]{inputenc} %

\usepackage[T1]{fontenc}    %
\usepackage{url}            %
\usepackage{booktabs}       %
\usepackage{amsfonts}       %
\usepackage{nicefrac}       %
\usepackage{microtype}      %
\usepackage{xcolor}         %

\usepackage{microtype}      %
\usepackage{appendix}
\usepackage{amssymb} 

\usepackage[ruled,vlined]{algorithm2e}
\SetKwInput{KwInit}{Initialization}
\LinesNumbered

\usepackage{amsmath}

\usepackage{amsthm}
\usepackage{nccmath}
\usepackage{empheq}
\usepackage{color,colortbl}
\usepackage{enumitem}
\usepackage{graphicx} %
\usepackage{empheq}

\usepackage{nicefrac}
\usepackage{booktabs}
\usepackage{multirow}
\usepackage{rotating}

\definecolor{mydarkblue}{rgb}{0,0.08,0.45}
\usepackage[colorlinks=true,
    linkcolor=mydarkblue,
    citecolor=mydarkblue,
    filecolor=mydarkblue,
    urlcolor=mydarkblue,
    pdfview=FitH]{hyperref}

\DeclareMathOperator*{\dom}{dom}

\DeclareMathOperator*{\argmin}{{arg\,min}}

\definecolor{myblue}{HTML}{D2E4FC}
\definecolor{Gray}{gray}{0.92}

\definecolor{certdual}{HTML}{2B6CB0}
\definecolor{certgrad}{HTML}{2F855A}
\definecolor{certrem}{HTML}{C05621}
\definecolor{primalstep}{HTML}{1F4E79}
\definecolor{dualstep}{HTML}{B8860B}
\newcommand{\certbox}[2]{{\setlength{\fboxsep}{2pt}\fcolorbox{#1}{#1!10}{$\displaystyle #2$}}}

\newtheorem{theorem}{Theorem}
\newtheorem{assumption}{Assumption}
\newtheorem{proposition}{Proposition}
\newtheorem{lemma}{Lemma}

\newtheorem{corollary}{Corollary}
\newtheorem{definition}{Definition}

\definecolor{burnin}{HTML}{bfe4bf}
\definecolor{linearphase}{HTML}{fedfc2}

\newtheorem{innercustomgeneric}{\customgenericname}
\providecommand{\customgenericname}{}
\newcommand{\newcustomtheorem}[2]{%
  \newenvironment{#1}[1]
  {%
   \renewcommand\customgenericname{#2}%
   \renewcommand\theinnercustomgeneric{##1}%
   \innercustomgeneric
  }
  {\endinnercustomgeneric}
}

\newcustomtheorem{requirement}{Requirement}

\usepackage{cleveref}
\crefname{assumption}{Assumption}{Assumptions}
\crefname{innercustomgeneric}{Requirement}{Requirements}
\crefname{theorem}{Theorem}{Theorems}
\crefname{algocfline}{line}{lines}
\Crefname{algocfline}{Line}{Lines}

\crefformat{assumption}{#2As.~#1#3}
\Crefformat{assumption}{#2As.~#1#3}

\usepackage{soul}

\usepackage{thm-restate}

\crefformat{equation}{(#2#1#3)}

\title{ABrA-GD: Adaptive Bregman Accelerated Gradient Descent for Relatively Smooth and (Strongly-)Convex Optimization}

\begin{document}

\maketitle

\begin{abstract}
We propose ABrA-GD (Adaptive Bregman Accelerated Gradient Descent), an adaptive algorithm for relatively smooth convex optimization. The algorithm is derived from a computable primal--dual certificate that guarantees progress by comparing the objective value with a lower bound on the minimum of a regularized objective. This certificate enables adaptation to both smoothness and geometry, using only the relative strong convexity constant as a problem-dependent input. We introduce the dual Bregman Length Distortion Factor (BLDF), which measures how dual Bregman lengths change under an anchor shift or rescaling. Under bounded local dual BLDF, ABrA-GD achieves an accelerated $O(1/k^2)$ rate for convex objectives and an accelerated linear rate for relatively strongly convex objectives.
\end{abstract}

\section{Setting and Introduction}
\label{sec:intro}

We consider optimization problems of the form
\begin{equation}\label{eq:phi_objective_boxed}
\phi_\star
=
\min_{x\in\mathcal X}\phi(x)
\;\coloneqq\;
f(x)+\mu D_d(x,x_0), \quad D_d(x,y):=d(x)-d(y)-\langle \nabla d(y),x-y\rangle.
\end{equation}
The term $D_d(x,y)$ is the Bregman divergence generated by \(d\)
\citep{bregman1967relaxation,censor1981iterative,censor1992proximal}. Throughout the paper we work
under the following hypotheses.

\begin{assumption}[Problem setting]\label{ass:setting}
\(d\) is a Legendre function, \(\mathcal X=\operatorname{cl}(\dom d)\), and
\(x_0\in\operatorname{int}(\dom d)\). The function \(f:\mathcal X\to\mathbb R\) is
convex, differentiable, and relatively \((L-\mu)\)-smooth with respect to \(d\)
\citep{bauschke2017descent,lu2018relatively}, with \(0\le\mu<L\).
\end{assumption}

Assumption~\ref{ass:setting} is equivalent to requiring \(\phi\) to be \(L\)-smooth and
\(\mu\)-strongly convex relative to \(d\) (\cref{app:defs,app:anchored_decomposition}). The domain of \(d\) determines the feasible set. The analysis uses interior minimizers, for which
the optimality conditions are stationarity equations. Affine equality constraints are also allowed: after a change of
coordinates, the problem satisfies the same assumptions
(\cref{app:affine_restriction}).

\paragraph{Acceleration in Euclidean optimization.}
For smooth convex objectives, accelerated gradient methods improve
the \(O(1/k)\) rate of gradient descent to \(O(1/k^2)\)
\citep{nesterov2004introductory}. Several viewpoints explain this
improvement, including estimate sequences, Lyapunov functions,
and the coupling of gradient and mirror-descent steps
\citep{wilson2021lyapunov,allen2014linear}.
Primal--dual interpretations express progress through a certificate
that combines an objective upper bound with a lower model
\citep{diakonikolas2019approximate}.

\paragraph{Beyond Euclidean smoothness.}
Relative smoothness replaces the Euclidean quadratic upper model
with a Bregman divergence
\citep{bauschke2017descent,lu2018relatively}.
This allows objectives whose gradients are not globally Lipschitz,
including log-determinant, entropy, and Poisson-type objectives
with singular behavior near the domain boundary.
Bregman gradient descent attains an \(O(1/k)\) rate in this setting.
Unlike the Euclidean case, this rate cannot be improved over the
full relatively smooth class \citep{dragomir2022optimal}.
Acceleration therefore requires additional structure.

\paragraph{Conditions for Bregman acceleration.}
Accelerated Bregman Proximal Gradient (ABPG) methods obtain
\(O(k^{-\gamma})\) rates under a Triangle Scaling Property
(TSP) with exponent \(\gamma\) and gain \(1\)
\citep{hanzely2021accelerated,savchuk2024accelerated}.
For standard geometries such as entropy and Burg, this exponent is at most \(1\), so
\emph{this guarantee gives no acceleration}.
The accelerated variants instead fix \(\gamma=2\) and backtrack on the gain, so their rate
depends on the gains accepted during the run.
\citet{chen2026accelerated} obtain an accelerated linear rate under a
generalized Cauchy--Schwarz condition, whose constant is an input of
their method.
\citet{liu2022dual} additionally assume relative strong convexity,
but their rate does not improve under this assumption.

\subsection*{Contributions}

\begin{enumerate}
    \item \textbf{A primal--dual formulation for acceleration
    (\cref{sec:envelope}).}
    We formulate a descent requirement using the Bregman--Moreau
    envelope and certify it through its dual problem. This framework
    provides a common starting point for the algorithm and its analysis;
    \cref{app:envelope_special_cases} recovers Bregman gradient descent and
    dual averaging from it.

    \item \textbf{ABrA-GD
    (\cref{sec:algorithm}).}
    We derive ABrA-GD (\cref{algo:bregman_adaptive}), in which one gradient serves two steps:
    A Bregman gradient step gives an upper bound on the objective, and a
    dual-averaging step gives a lower bound on the minimum of a regularized objective.
    ABrA-GD backtracks on smoothness and acceleration parameters, and we
    establish a generic convergence bound in terms of the accepted acceleration
    parameters (\cref{thm:generic_accelerated_rate_Mk}).

    \item \textbf{A local measure of geometric distortion
    (\cref{sec:bldf}).}
    We introduce the dual Bregman Length Distortion Factor (BLDF)
    to compare the progress of gradient and dual-averaging steps. We give explicit
    formulas for standard separable kernels and relate this local
    condition to the triangle-scaling conditions.
    The algorithm adapts without evaluating the BLDF.

    \item \textbf{Accelerated rates under bounded distortion
    (\cref{sec:accelerated-rate}).}
    Under bounded local distortion, \textit{a condition implied by a uniform
    triangle-scaling gain (\cref{app:tsp})}, we obtain an accelerated \(O(1/k^2)\)
    rate for convex objectives and an accelerated linear rate for
    relatively strongly convex objectives. ABrA-GD's only problem-constant
    input is \(\mu\). To our knowledge, it is the
    first method with these inputs that has an accelerated linear rate
    for this class.
\end{enumerate}

\paragraph{Limitations.}
The accelerated analysis does not cover an additional nonsmooth term, unlike the ABPG analyses of
\citet{hanzely2021accelerated,savchuk2024accelerated}. 
The envelope descent requirement extends to composite objectives
(\cref{app:composite_subgradient}), but an accelerated BLDF analysis remains open.
\section{The Primal--Dual Envelope Framework}
\label{sec:envelope}

For \(\eta>0\), define the regularized objective
\begin{equation} \label{eq:enveloppe}
\phi_\star^{\eta} := \min_x \phi^\eta(x),\;\; \phi^\eta(x) :
=
\phi(x)+\frac{1}{\eta}D_d(x,x_0)
=
f(x)+\left(\mu+\tfrac{1}{\eta}\right)D_d(x,x_0),
\qquad \eta>0.
\end{equation}
The value \(\phi^\eta_\star\) is the Bregman--Moreau envelope of \(\phi\) anchored at \(x_0\) \citep{bauschke2018regularizing}, which approaches \(\phi_\star\)
as \(\eta\to\infty\). Consider the implicit envelope descent scheme (\cref{algo:implicit_descent}).

\begin{algorithm}[H]
\caption{Implicit envelope descent (conceptual algorithm, used only as convergence tool)}\label{algo:implicit_descent}
\KwIn{$\phi, x_0$. \textbf{Initialization:} $k\leftarrow 0$, $\eta_0 \leftarrow 0$.}
\For{$k = 0,1,2,\dots$}{
    Choose $\eta_{k+1}\ge\eta_k$ and find $x_{k+1}$ such that $\phi(x_{k+1})\le\phi_\star^{\eta_{k+1}}$ (defined in \cref{eq:enveloppe}).
}
\end{algorithm}

Each iterate must satisfy \(\phi(x_{k+1})\le\phi_\star^{\eta_{k+1}}\), but it need not minimize the
regularized objective (\cref{fig:implicit_envelope}). The scheme does not specify how to compute
\(x_{k+1}\) or verify this requirement, since the envelope value is itself defined by a
minimization. It does, however, give an immediate error bound.

\begin{theorem}\label{thm:implicit_descent_rate}
    Under \cref{ass:setting}, \cref{algo:implicit_descent} with inputs $\phi,\,x_0$ satisfies
    \begin{equation} \label{eq:rate_implicit}
        \phi(x_{k})-\phi_\star \leq \eta_k^{-1}D_d(x_\star,x_0).
    \end{equation}
\end{theorem}

To prove the bound, evaluate the regularized objective at \(x_\star\)
(\cref{app:rates}). The growth of \(\eta_k\) is therefore the central choice. Faster growth gives a
faster rate in \cref{eq:rate_implicit}, but it also lowers \(\phi_\star^{\eta_k}\) toward
\(\phi_\star\), so each iteration must make more progress.

\paragraph{Link with proximal methods.} The proximal point method \citep{lemarechal1997practical} centers each subproblem at the current
iterate, and Accelerated Hybrid Proximal Extragradient \citep{monteiro2013accelerated} and Catalyst
\citep{lin2015universal,lin2018catalyst} solve these subproblems approximately. In
\cref{algo:implicit_descent}, the anchor \(x_0\) stays fixed and no subproblem is solved.

\subsection{Envelope Dual Problem and Lower Bound}
\label{sec:dual_problem}

The envelope value \(\phi_\star^{\eta_{k+1}}\) is defined by a minimization, so the condition
\(\phi(x_{k+1})\le\phi_\star^{\eta_{k+1}}\) cannot be checked directly. We certify it through the
dual problem and provide a sufficient condition for the requirement of
\cref{algo:implicit_descent}. For any \(\lambda\), the Fenchel--Young inequality gives the affine lower bound
\begin{equation}\label{eq:affine_lower_model}
\ell_\lambda(x)
:= \langle x,\lambda\rangle-f^*(\lambda)
\le f(x),
\end{equation}
where \(f^*\) is the Fenchel conjugate of \(f\), with equality if and only if
\(\lambda=\nabla f(x)\). Adding the regularization defines the dual value and its
associated primal point:
\begin{align}
\mathcal D^\eta(\lambda)
&:= \min_x \left\{
\ell_\lambda(x)+(\mu+\eta^{-1})D_d(x,x_0)
\right\},
\label{eq:dual_function}\\
z^\eta(\lambda)
&:= \arg\min_x \left\{
\ell_\lambda(x)+(\mu+\eta^{-1})D_d(x,x_0)
\right\}
= \nabla d^*\!\left(
\nabla d(x_0)-\frac{\lambda}{\mu+\eta^{-1}}
\right).
\label{eq:primal_map}
\end{align}
By weak duality, we have $\ell_\lambda\le f$, and therefore $\mathcal D^\eta(\lambda)\le\phi_\star^\eta\le\phi^\eta(x)$
for every $x$.
Although $z^\eta(\lambda)$ is computable without evaluating $f^*(\lambda)$,
the dual value $\mathcal D^\eta(\lambda)$ generally is not. The next section shows how to certify \(\phi(x_{k+1})\le\mathcal D^{\eta_{k+1}}(\lambda_{k+1})\) without evaluating \(f^*\).
\section{Deriving ABrA-GD and Its Accepted-Parameter Rate}
\label{sec:algorithm}

\begin{figure}[p]
\begin{minipage}{\textwidth}
\setlength{\abovedisplayskip}{1pt}\setlength{\belowdisplayskip}{1pt}%
\setlength{\abovedisplayshortskip}{0pt}\setlength{\belowdisplayshortskip}{2pt}%

\begin{minipage}[t]{0.475\textwidth}
\centering
\includegraphics[width=\linewidth]{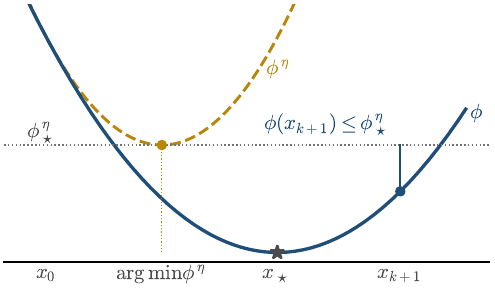}
\vspace{-5ex}
\caption{Envelope descent requires that
\(\phi(x_{k+1})\le\phi_\star^\eta\), without minimizing
the regularized objective \(\phi^\eta\).}
\label{fig:implicit_envelope}
\end{minipage}%
\hfill
\begin{minipage}[t]{0.475\textwidth}
\centering
\includegraphics[width=\linewidth]{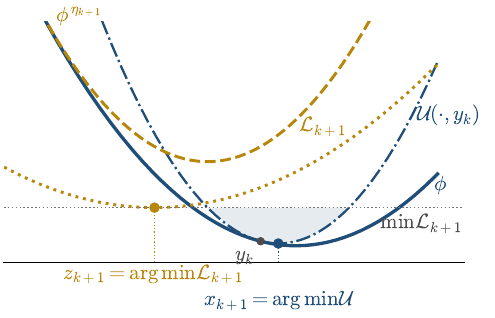}
\vspace{-5ex}
\caption{The model minimizers are \(x_{k+1}\) and \(z_{k+1}\).
An upper model \cref{eq:abra_upper_model} whose minimum value is below \(\min_x\mathcal L_{k+1}(x)\) \cref{eq:computable_lowerbound}
certifies envelope descent.}
\label{fig:model_sandwich}
\end{minipage}

\smallskip

\begin{algorithm}[H]
\caption{\textbf{ABrA-GD} --- \textbf{A}daptive \textbf{Br}egman \textbf{A}ccelerated \textbf{G}radient \textbf{D}escent}
\label{algo:bregman_adaptive}
\KwIn{$x_0$, $\mu\ge0$, $\phi$, initial estimates $L_{\mathrm{cur}}=L_0>0$ and $M_{-1}>0$ (default \(L_0=1\), \(M_{-1}=L_0\))}
\KwInit{$x_0^{\text{best}}=z_0=x_0$, $\lambda_0=0$, $\eta_0=0$, $\phi^{\rm best}_0=\phi(x_0),\,  t_0 = 1;$}

\For{$k=0,1,2,\dots$}{
    $M_k\leftarrow 
    \max
    \{M_{k-1}/4,\mu\}, \quad L_{\mathrm{cur}}\leftarrow \max\{ L_{\mathrm{cur}}/2,\,\mu\}, \quad \phi^{\rm best}_{k+1}\leftarrow \phi^{\rm best}_k, \quad x^{\rm best}_{k+1}\leftarrow x^{\rm best}_k$\;
    \Repeat{$
        \phi^{\rm best}_{k+1}
        \le \mathcal{L}_{k+1}(z_{k+1}) :=
        (1-t_k)\phi^{\rm best}_k
        +t_k\phi^{\rm low}_k
        -(\mu+\eta_{k+1}^{-1})D_d(z_k,z_{k+1})
    $\label{line:dual_decrease}}{
       \[
        \begin{aligned}
        M_k &\leftarrow 2M_k,
        &
        t_{k\geq 1}
        &=
        \min\left\{
        1,\,
        \frac{2(\eta_k^{-1}+\mu)}
        {\eta_k^{-1}+\sqrt{\eta_k^{-2}+4M_k(\eta_k^{-1}+\mu)}}
        \right\},
        \\
        \eta_{k+1}^{-1}
        &=M_kt_k^2-\mu,
        &
        \tau_k
        &=
        \frac{t_k-\mu/\sqrt{M_kL_{\mathrm{cur}}}}
             {1-\mu/\sqrt{M_kL_{\mathrm{cur}}}}.
        \end{aligned}
        \]
        \[
        (x_{k+1},z_{k+1},\lambda_{k+1},L_{\mathrm{cur}},\phi^{\rm low}_k)
        =
        \textsc{Breg-PD-step}
        (x_{k}^{\text{best}},z_k,\lambda_k,\eta_{k+1},L_{\mathrm{cur}},t_k,\tau_k).
        \]
        \lIf{$\phi^{\rm best}_{k+1}>\phi(x_{k+1})$}{
              $\phi^{\rm best}_{k+1}\leftarrow\phi(x_{k+1})$,\quad
              $x^{\rm best}_{k+1}\leftarrow x_{k+1}$
            }
    }
}
\end{algorithm}

\begin{algorithm}[H]
\caption{\textsc{Breg-PD-step}$(x,z,\lambda,\eta,L_{\mathrm{cur}},t,\tau)$}
\label{algo:accstep}
\LinesNotNumbered
\KwIn{$x,z,\lambda,\eta,L_{\mathrm{cur}},t,\tau$}
\[
y=(1-\tau)x+\tau z,
\qquad
\phi^{\mathrm{low}}
=
\phi(y)
+
\langle\nabla\phi(y),z-y\rangle
+
\mu D_d(z,y).
\]
\textbf{Update}
\begin{align}
(x_{\mathrm{new}},L_{\mathrm{cur}})
&=
\textsc{Backtracking-GD}
\bigl(y,\phi(y),\nabla\phi(y),L_{\mathrm{cur}}\bigr),
\tag{Primal Step}
\\
\lambda_{\mathrm{new}}
&=
(1-t)\lambda
+
t\Big(\nabla\phi(y)-\mu(\nabla d(y)-\nabla d(x_0))\Big),
\tag{Dual Step}
\\
z_{\mathrm{new}}
&=
\arg\min_u
\left\{
\langle \lambda_{\mathrm{new}},u\rangle
+
\left(\mu+\eta^{-1}\right)D_d(u,x_0)
\right\}.
\tag{Primal Map}
\end{align}
\Return $(x_{\mathrm{new}},z_{\mathrm{new}},\lambda_{\mathrm{new}},L_{\mathrm{cur}},\phi^{\mathrm{low}})$\;
\end{algorithm}

\begin{algorithm}[H]
\caption{\textsc{Backtracking-GD}$(y,\phi(y),\nabla\phi(y),L_{\mathrm{cur}})$}
\label{algo:backtracking_gradient}
\KwIn{$y$, $\phi(y)$, $\nabla\phi(y)$, trial constant $L_{\mathrm{cur}}>0$.

\KwInit{$L_{\mathrm{new}}=L_{\mathrm{cur}}/2$}}
\Repeat{$
\phi(x_{\mathrm{new}})
\le
\phi(y)
+
\langle \nabla\phi(y),x_{\mathrm{new}}-y\rangle
+
L_{\mathrm{new}} D_d(x_{\mathrm{new}},y)
=\phi(y)-L_{\mathrm{new}}D_d(y,x_{\mathrm{new}})$\label{line:primal_decrease}
}{
\[
L_{\mathrm{new}}\leftarrow 2L_{\mathrm{new}},\quad
x_{\mathrm{new}}
=
\arg\min_x
\left\{
\langle \nabla\phi(y),x\rangle
+
L_{\mathrm{new}}D_d(x,y)
\right\},
\]
}
\Return $(x_{\mathrm{new}},L_{\mathrm{new}})$\;
\end{algorithm}

\end{minipage}
\end{figure}

\Cref{algo:implicit_descent} is conceptual. The dual formulation \cref{eq:dual_function} says
what to certify, \(\phi(x_{k+1})\le\mathcal D^{\eta_{k+1}}(\lambda_{k+1})\), but not how to
construct \(x_{k+1}\), \(\lambda_{k+1}\) (hence \(z_{k+1}=z^{\eta_{k+1}}(\lambda_{k+1})\)), or
\(\eta_{k+1}\). ABrA-GD constructs all of them, with \(\eta_k\to\infty\), and maintains the
invariant
\begin{equation}\label{eq:abra_invariant}
\phi_k^{\mathrm{best}}
\le \mathcal D^{\eta_k}(\lambda_k)
\le \phi_\star^{\eta_k}.
\end{equation}
Each trial of an iteration evaluates one gradient at a point \(y_k\) between
\(x_k^{\mathrm{best}}\) and \(z_k\). This gradient gives a primal candidate \(x_{k+1}\) and a dual
update \(\lambda_{k+1}\). A computable lower bound then tests whether the invariant holds at step
\(k+1\); if not, backtracking increases \(M_k\) and repeats the trial.

\paragraph{1. Common gradient evaluation.}
Both model updates use one common gradient evaluation at the point $y_k$, defined below. Here \(x_k^{\mathrm{best}}\) is the best retained point and \(z_k\)
is the current primal map; the analysis specifies \(\tau_k\)
(\cref{lem:bound_acc_generic}):
\begin{equation}\label{eq:abra_evaluation_point}
y_k=(1-\tau_k)x_k^{\mathrm{best}}+\tau_k z_k.
\end{equation}

\paragraph{2. Primal upper model.}
The upper model at \(y_k\) is
\begin{equation}\label{eq:abra_upper_model}
\mathcal U(x,y_k)
:= \phi(y_k)
+ \langle \nabla\phi(y_k),x-y_k\rangle
+ L_{\mathrm{cur}}D_d(x,y_k).
\end{equation}
Since \(\phi\) is \(L\)-smooth relative to \(d\) (\cref{ass:setting}),
\(\phi\le\mathcal U(\cdot,y_k)\) when \(L_{\mathrm{cur}}\ge L\). The \emph{Primal Step} in
\cref{algo:accstep} minimizes this model:
\begin{equation}\label{eq:abra_primal_candidate}
x_{k+1}=\arg\min_x \mathcal U(x,y_k).
\end{equation}
\textsc{Backtracking-GD} (\cref{algo:backtracking_gradient}) adjusts \(L_{\mathrm{cur}}\) until
\(\phi(x_{k+1})\le\mathcal U(x_{k+1},y_k)\), and terminates whenever \(L_{\mathrm{cur}}\ge L\).

\paragraph{3. Dual step and primal map.}
The \emph{Dual Averaging Step} updates $\lambda_k$ with the same gradient:
\begin{equation}\label{eq:abra_dual_update}
\lambda_{k+1}
=(1-t_k)\lambda_k+t_k\nabla f(y_k),\quad t_k\in[0,1].
\end{equation}
Here \(\nabla f(y_k)=\nabla\phi(y_k)-\mu(\nabla d(y_k)-\nabla d(x_0))\),
so no additional objective-gradient evaluation is needed.
The relation to dual best responses and conditional gradients
\citep{bach2015duality} is discussed in \cref{sec:dual_avg}.
The \emph{Primal Map} then gives
\begin{equation}\label{eq:abra_lower_minimizer}
z_{k+1}=z^{\eta_{k+1}}(\lambda_{k+1}).
\end{equation}
Both updates in \cref{eq:abra_dual_update,eq:abra_lower_minimizer} appear in \cref{algo:accstep}.

\paragraph{4. Parameters \(t_k\), $\eta_k$, and \(M_k\).}
For \(k\ge1\), the weight \(t_k\) and regularization \(\eta_{k+1}\) satisfy
\begin{equation}\label{eq:abra_parameters}
\eta_{k+1}^{-1}=(1-t_k)\eta_k^{-1},
\qquad
\mu+\eta_{k+1}^{-1}=M_kt_k^2 .
\end{equation}
The first identity matches the weight on the previous dual model, enabling
the recursive certificate below. Together, the identities give
\(M_kt_k^2=(1-t_k)(\mu+\eta_k^{-1})+t_k\mu\), the quadratic relation
in Nesterov's method \citep{nesterov2004introductory}.
ABrA-GD selects \(M_k\) by backtracking and computes the positive root \(t_k\).
The accepted \(M_k\) controls the rate; see \cref{thm:generic_accelerated_rate_Mk}.

\paragraph{5. A computable certificate.}
Evaluating \(\mathcal D^{\eta_{k+1}}(\lambda_{k+1})\) requires \(f^*(\lambda_{k+1})\), which is not
tractable in general. This step builds a computable lower bound
\(\mathcal L_{k+1}(z_{k+1})\le\mathcal D^{\eta_{k+1}}(\lambda_{k+1})\). The exit condition of
\cref{algo:bregman_adaptive} checks \(\phi_{k+1}^{\mathrm{best}}\le\mathcal L_{k+1}(z_{k+1})\), which
gives \cref{eq:abra_invariant} at step \(k+1\).

\Cref{lem:pd_lower_step} gives the construction; we sketch it. The affine bound \(\ell_\lambda\) of
\cref{eq:affine_lower_model} is concave in \(\lambda\), so the dual update \cref{eq:abra_dual_update}
and the first relation in \cref{eq:abra_parameters} split \(\mathcal D^{\eta_{k+1}}(\lambda_{k+1})\)
into two terms. Relative strong convexity gives the \textcolor{certdual}{first}, and the Fenchel
equality with the three-point identity gives the \textcolor{certgrad}{second}:
\begin{equation}\label{eq:abra_certificate_split}
\begin{aligned}
\mathcal D^{\eta_{k+1}}(\lambda_{k+1})
\ge\min_x\Big\{&(1-t_k)\certbox{certdual}{\mathcal D^{\eta_k}(\lambda_k)\!+\!(\mu+\eta_k^{-1})D_d(x,z_k)}\\
&+t_k\certbox{certgrad}{\phi(y_k)\!+\!\langle\nabla\phi(y_k),x-y_k\rangle\!+\!\mu D_d(x,y_k)}\Big\}.
\end{aligned}
\end{equation}
By the invariant \cref{eq:abra_invariant} at step \(k\), \(\mathcal D^{\eta_k}(\lambda_k)\ge\phi_k^{\mathrm{best}}\).
Replacing \(\mathcal D^{\eta_k}(\lambda_k)\) by \(\phi_k^{\mathrm{best}}\) in the first box of
\cref{eq:abra_certificate_split} therefore gives
\(\mathcal D^{\eta_{k+1}}(\lambda_{k+1})\ge\min_x\mathcal L_{k+1}(x)\), where \(\mathcal L_{k+1}\) is the
lower model (\cref{fig:model_sandwich})
\begin{equation}\label{eq:abra_lower_model}
\hspace{-1.5em}
\mathcal L_{k+1}(x)
:=(1-t_k)\certbox{certdual}{\phi_k^{\mathrm{best}}\!+\!(\mu+\eta_k^{-1})D_d(x,z_k)}
+t_k\certbox{certgrad}{\phi(y_k)\!+\!\langle\nabla\phi(y_k),x-y_k\rangle\!+\!\mu D_d(x,y_k)}.
\end{equation}
Both boxes use only \(\phi_k^{\mathrm{best}}\), \(z_k\) and one gradient at \(y_k\), so
\(\mathcal L_{k+1}\) is computable. Finally, \cref{lem:pd_lower_step} shows that \(\mathcal L_{k+1}\) is minimized at \(z_{k+1}\). Writing \(\phi_k^{\mathrm{low}}\) for the \textcolor{certgrad}{second term} of \cref{eq:abra_lower_model}
at \(x=z_k\),
we obtain the right-hand side of the exit condition of \cref{algo:bregman_adaptive}:
\begin{equation} \label{eq:computable_lowerbound}
\mathcal L_{k+1}(z_{k+1})
=
(1-t_k)\phi_k^{\mathrm{best}}
+t_k\phi_k^{\mathrm{low}}
-(\mu+\eta_{k+1}^{-1})D_d(z_k,z_{k+1}).
\end{equation}

\subsection*{Convergence in terms of the accepted parameters} 
\label{sec:analysis}

The exit condition preserves the envelope descent requirement. The following theorem expresses
the resulting rate in terms of the accepted acceleration parameters \(M_k\), without imposing an
additional condition on the geometry. \textit{The theorem does not bound the accepted parameters themselves, nor show that the inner loop
terminates}. \Cref{sec:bldf,sec:accelerated-rate} give accelerated rates under bounded local
distortion.

\begin{restatable}[A posteriori convergence of ABrA-GD]{theorem}{GenericAcceleratedRateMk}
\label{thm:generic_accelerated_rate_Mk}
Under \cref{ass:setting}, assume that the inner loop of \cref{algo:bregman_adaptive} exits at every
iteration. Let
\(M_k>\mu\) and choose \(t_k\in(0,1]\) by
\[
\mu+\eta_{k+1}^{-1}=M_kt_k^2
\quad(k\ge0),
\qquad
\eta_{k+1}^{-1}=(1-t_k)\eta_k^{-1}
\quad(k\ge1),
\qquad
t_0=1,
\quad
\eta_1^{-1}=M_0-\mu .
\]
Then, for any \(x_\star\in\argmin_x\phi(x)\),
\[
\phi_{k+1}^{\mathrm{best}}\le \phi_k^{\mathrm{best}},
\qquad
\phi_{k+1}^{\mathrm{best}}\le \phi_\star^{\eta_{k+1}},
\qquad
\phi_{k+1}^{\mathrm{best}}-\phi_\star
\le
\eta_{k+1}^{-1}D_d(x_\star,x_0).
\]
Moreover, the decay of \(\eta_k^{-1}\) is controlled by the acceleration parameters \(M_k\):
\[
\begin{array}{c|c|c}
\textbf{Case}
& \textbf{No uniform bound on \(M_k\)}
& \textbf{If \(M_i\le M\) for all \(i\ge K\)} \\
\hline
\mu=0
&
\displaystyle
\eta_k^{-1}
\le
\frac{4}{
\left(\sum_{i=0}^{k-1}M_i^{-1/2}\right)^2}
&
\displaystyle
\eta_k^{-1}
\le
\frac{\eta_K^{-1}}{
\left(
1+
\frac{k-K}{2\sqrt{M\eta_K}}
\right)^2}
\\[1.4em]
\mu>0
&
\displaystyle
\eta_k^{-1}
\le
(M_0-\mu)
\prod_{i=1}^{k-1}
\left(1-\sqrt{\frac{\mu}{M_i}}\right)
&
\displaystyle
\eta_k^{-1}
\le
\eta_K^{-1}
\left(
1-\sqrt{\frac{\mu}{M}}
\right)^{k-K}.
\end{array}
\]
\end{restatable}
\section{Dual Bregman Length Distortion}
\label{sec:bldf}

We identify the local geometric distortion in the comparison between the gradient and
dual-averaging steps. We define the dual BLDF, relate it to local norms, and give formulas for
standard kernels. \Cref{sec:accelerated-rate} uses this condition to obtain accelerated rates and
compares them with ABPG.

\subsection{Comparing the gradient and dual-averaging steps (convex case)}

The exit condition requires \(\phi(x_{k+1})\le\mathcal L_{k+1}(z_{k+1})\). We compare how far the
upper model and the lower bound decrease from the common value \(\phi(y_k)\).
Here \(\mu=0\) and \(\tau_k=t_k\). The primal step
(line~\ref{line:primal_decrease} of \cref{algo:backtracking_gradient}) and the exit test
(line~\ref{line:dual_decrease} of \cref{algo:bregman_adaptive}) give
\begin{align*}
\phi(x_{k+1})\le\phi(y_k)-\textcolor{primalstep}{\underbrace{L_{\mathrm{cur}}D_d(y_k,x_{k+1})}_{\text{primal decrease}}},\quad 
\mathcal L_{k+1}(z_{k+1})\ge\phi(y_k)-\textcolor{dualstep}{\underbrace{\eta_{k+1}^{-1}D_d(z_k,z_{k+1})}_{\text{dual decrease}}},
\end{align*}
where the second inequality uses \(y_k=(1-t_k)x_k^{\mathrm{best}}+t_kz_k\) and convexity,
\((1-t_k)\phi_k^{\mathrm{best}}+t_k\phi_k^{\mathrm{low}}\ge\phi(y_k)\) (\cref{lem:bound_acc_convex}).
The exit condition is therefore guaranteed to hold when the upper model decreases at least as much as the lower
bound:
\begin{equation}\label{eq:decrease_comparison}
\textcolor{dualstep}{\eta_{k+1}^{-1}D_d(z_k,z_{k+1})}\le \textcolor{primalstep}{L_{\mathrm{cur}}D_d(y_k,x_{k+1})} \quad \Rightarrow \quad \phi(x_{k+1}) \leq \mathcal L_{k+1}(z_{k+1}).
\end{equation}
When
\(d=\tfrac12\|\cdot\|^2\), both decreases are multiples of \(\|\nabla\phi(y_k)\|^2\):
\(\textcolor{primalstep}{L_{\mathrm{cur}}^{-1}\|\nabla\phi(y_k)\|^2/2}\) and
\(\textcolor{dualstep}{t_k^2\eta_{k+1}\|\nabla\phi(y_k)\|^2/2}\). Equality in \cref{eq:decrease_comparison} gives
\(\eta_{k+1}^{-1}=L_{\mathrm{cur}}t_k^2\), the Nesterov coupling (\cref{app:euclidean_nesterov}).
For a general \(d\), the two decreases are Bregman divergences taken at different points and with
different step lengths, so they cannot be compared term by term.
In dual coordinates, however, with \(g_k=\nabla\phi(y_k)\), the two steps are segments along the same direction:
\[
\nabla d(x_{k+1})=\nabla d(y_k)-L_{\mathrm{cur}}^{-1}g_k,
\qquad
\nabla d(z_{k+1})=\nabla d(z_k)-t_k\eta_{k+1}g_k.
\]
Since \(D_d(x,y)=D_{d^*}(\nabla d(y),\nabla d(x))\), the condition in
\cref{eq:decrease_comparison} reads
\[
\textcolor{dualstep}{\eta_{k+1}^{-1}D_{d^*}\big(\nabla d(z_k)-t_k\eta_{k+1}g_k,\,\nabla d(z_k)\big)}
\le
\textcolor{primalstep}{L_{\mathrm{cur}}D_{d^*}\big(\nabla d(y_k)-L_{\mathrm{cur}}^{-1}g_k,\,\nabla d(y_k)\big)}.
\]
Both sides measure the Bregman length of a segment along the same direction \(-g_k\), from
different anchors, \(\nabla d(z_k)\) and \(\nabla d(y_k)\), and with different scales,
\(t_k\eta_{k+1}\) and \(L_{\mathrm{cur}}^{-1}\).

\paragraph{Strongly convex case.} When \(\mu>0\), the two dual segments point toward
\(\nabla d(q_{k+1})\), with \(q_{k+1}\) as in \cref{lem:bound_acc_strongly_convex}, and are no longer parallel (\cref{fig:two_lengths}(b)). The proof of
\cref{prop:strongly_convex_rate_dual_distortion} makes two comparisons, each between two segments
with the same anchor and direction but different scales: from \(\nabla d(y_k)\) to
\(\nabla d(x_{k+1})\) and to \(\nabla d(q_{k+1})\), and likewise from \(\nabla d(z_k)\).

\begin{figure}[t]
\centering
\includegraphics[width=0.72\textwidth]{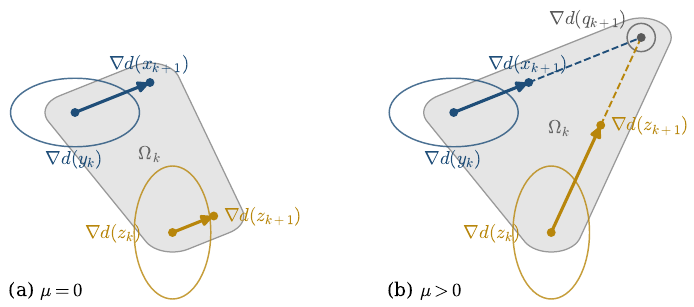}
\caption{The two steps in dual coordinates \(u=\nabla d(x)\) for the Burg kernel. The ellipses
are the unit balls \(\{\xi:\tfrac12\langle\xi,\nabla^2d^*(w)\,\xi\rangle\le1\}\) of the local dual
norm (\cref{eq:primal_dual_norm}) at the anchors \(w\), drawn at a common scale; they change
across the region, and \(\rho_{d^*}\) bounds this change on the shaded convex set, which
contains \(\Omega_k\) (\cref{def:omega_k}). \textbf{(a)} When \(\mu=0\), the primal
and dual steps are parallel displacements along \(-\nabla\phi(y_k)\), from different anchors and
with different scales. \textbf{(b)} When \(\mu>0\), \(\nabla d(x_{k+1})\) and \(\nabla d(z_{k+1})\)
lie on two segments whose extensions meet at \(\nabla d(q_{k+1})\); the shaded set contains the
triangle they form with \(\nabla d(y_k)\) and \(\nabla d(z_k)\).}
\label{fig:two_lengths}
\end{figure}

\subsection{The dual Bregman length distortion factor}\label{sec:dual_bldf_def}

In both cases, the comparison is between Bregman lengths of segments along a common direction,
after a change of anchor, of scale, or both. The \textit{local dual Bregman length distortion
factor} bounds this change.

\begin{definition}[Local Dual Bregman Length Distortion Factor (\textbf{Dual BLDF})] \label{def:dual_bldf}
Let $\Omega\subset \operatorname{int}(\dom d^*)$, $s>0$, and $\gamma>0$. The \emph{dual BLDF} of $d^*$ on $\Omega$ is
\[
c_{d^*}(s,\Omega;\gamma)
:=
\sup_{u_1,u_2,v\neq 0}
\left(
\frac{D_{d^*}(u_2+s v,u_2)}
{s^\gamma D_{d^*}(u_1+v,u_1)}
\right)^{1/2}
\quad \text{s.t.}\quad
\begin{cases}
u_1+[0,1]v\subset \Omega,\\
u_2+[0,s]v\subset \Omega.
\end{cases}
\]
\end{definition}

The numerator uses the displacement $sv$ at anchor $u_2$, while the denominator uses $v$ at
anchor $u_1$. At $\gamma=2$, the factor $s^2$ removes the quadratic scaling, leaving the change
due to the geometry. Our convergence analysis uses \(\gamma=2\); other exponents serve only the
comparison with triangle scaling (\cref{app:tsp}). If $d^*$ is $L_{d^*}$-smooth and $\mu_{d^*}$-strongly convex on $\Omega$, then
$c_{d^*}(s,\Omega;2)\le\sqrt{L_{d^*}/\mu_{d^*}}$ (\cref{lem:bldf_smooth_strongly_convex}), so our rate then matches that of
\citet{auslender2006interior} up to a constant. If $d^*$ is twice differentiable, the dual BLDF is
bounded by the variation of the dual local norm over $\Omega$ (\cref{prop:dual_bldf_twice_diff}).

\begin{definition}[Dual norm distortion factor]
Let $\Omega\subset \operatorname{int}(\dom d^*)$, and assume that $d^*$ is twice differentiable on $\Omega$. The \emph{dual norm distortion factor} of $d^*$ on $\Omega$ is
\[
\rho_{d^*}(\Omega)
:=
\sup_{w_1,w_2\in\Omega}\;
\sup_{\xi\neq 0}
\left(\frac{\langle\xi,\nabla^2d^*(w_2)\,\xi\rangle}{\langle\xi,\nabla^2d^*(w_1)\,\xi\rangle}\right)^{1/2}
=
\underbrace{\sup_{w_1,w_2\in\Omega}
\max_{1\le i\le n}
\left(
\frac{(d_i^*)''((w_2)_i)}{(d_i^*)''((w_1)_i)}
\right)^{1/2}}_{\text{if \(d^*\) is separable (\cref{cor:rho_separable})}}.
\]
\end{definition}

The factor $\rho_{d^*}(\Omega)$ measures the largest change in local norm across the region,
after taking the supremum over directions. It has no scaling parameter and bounds the dual BLDF
at $\gamma=2$.

\begin{restatable}{proposition}{DualBLDFTwiceDiff}\label{prop:dual_bldf_twice_diff}
Assume that $d^*$ is twice differentiable on $\Omega\subset \operatorname{int}(\dom d^*)$. Then, for every $s>0$,
\[
c_{d^*}(s,\Omega;2)\le \rho_{d^*}(\Omega).
\]
\end{restatable}

If \(\Omega\) is open and \(\nabla^2d^*\) is continuous and positive definite on
\(\Omega\), then \(c_{d^*}(s,\Omega;2)=\rho_{d^*}(\Omega)\) for every \(s>0\).
More generally, \(c_{d^*}(s,\Omega;\gamma)=s^{1-\gamma/2}\rho_{d^*}(\Omega)\)
for every \(s>0\) and \(\gamma>0\) (\cref{prop:dual_bldf_equality}).

The rates of \cref{sec:accelerated-rate} use \(\rho_{d^*}\). \Cref{prop:dual_bldf_twice_diff} gives simple expressions for common geometries
(\cref{tab:rho_examples}). For instance, for the Burg entropy,
\(\rho_{d^*}(\nabla d(\mathcal X))\le b/a\) whenever \(a\le x\le b\) for all \(x\in\mathcal X\).
\section{Accelerated Rates under Bounded Distortion}
\label{sec:accelerated-rate}

The accepted-parameter bound of \cref{thm:generic_accelerated_rate_Mk} becomes an accelerated rate when the accepted
\(M_k\) remain bounded. We now bound them using the local distortion
\(\rho_k\) of each iteration.

\begin{definition}[Local distortion]\label{def:omega_k}
Let \(\Omega_k\) be the convex hull of the dual segments of all trials of the \(M_k\)-search in
iteration \(k\) of \cref{algo:bregman_adaptive}. The \emph{local distortion} of iteration \(k\) is
\(\rho_k:=\rho_{d^*}(\Omega_k)\).
\end{definition}

\subsection{Convergence and iteration complexity}

Using the dual BLDF, \cref{prop:convex_rate_dual_distortion,prop:strongly_convex_rate_dual_distortion}
show that every trial with \(M_k\ge\rho_k^2L\) (\(\mu=0\)), or \(M_k\ge\rho_k^4L\) (\(\mu>0\)), passes
the exit test of \cref{algo:bregman_adaptive}. Hence, if \(\rho_k\le\rho\) from some iteration \(K\)
on, backtracking eventually keeps \(M_k\) below \(4\rho^2L\), respectively \(4\rho^4L\).

\begin{restatable}[Accelerated rate under bounded distortion]{theorem}{EventualBLDFRate}
\label{thm:eventual_bldf_rate}
Under \cref{ass:setting}, assume \(d^*\) is twice differentiable, and let \(\rho_k\) be as in
\cref{def:omega_k}. If there exist \(K\ge 0\) and \(\rho\) such that the inner loop of
\cref{algo:bregman_adaptive} exits at every iteration \(k<K\) and \(\rho_k\le\rho\) for all
\(k\ge K\), then it exits at every iteration, and there is a finite \(\bar K\ge K\) such that, for all
\(k\ge\bar K\), we have \(M_k\le 4\rho^2L\) (\(\mu=0\)), respectively \(M_k\le 4\rho^4L\) (\(\mu>0\)). By \cref{thm:generic_accelerated_rate_Mk}, this implies
\(\phi_k^{\mathrm{best}}-\phi_\star\le\eta_k^{-1}D_d(x_\star,x_0)\), with
\[
\eta_k^{-1}\le\frac{\eta_{\bar K}^{-1}}{\big(1+\frac{k-\bar K}{4\rho\sqrt{L\eta_{\bar K}}}\big)^2}
\;\textcolor{gray}{(\mu=0)},
\qquad
\eta_k^{-1}\le\eta_{\bar K}^{-1}\Big(1-\sqrt{\tfrac{\mu}{4\rho^4L}}\Big)^{k-\bar K}
\;\textcolor{gray}{(\mu>0)}.
\]
\end{restatable}

\begin{restatable}[Iteration complexity]{corollary}{IterationComplexity}
\label{cor:iteration_complexity}
Under \cref{thm:eventual_bldf_rate}, \(\phi_k^{\mathrm{best}}-\phi_\star\le\varepsilon\) after
\(N\) iterations from \(\bar K\):
\[
N=O\!\left(\rho\sqrt{\frac{L\,D_d(x_\star,x_0)}{\varepsilon}}\right)\quad(\mu=0),
\qquad
N=O\!\left(\rho^2\sqrt{\frac{L}{\mu}}\,
\log\frac{\eta_{\bar K}^{-1}D_d(x_\star,x_0)}{\varepsilon}\right)\quad(\mu>0).
\]
\end{restatable}

When \(\rho=1\), these bounds match the iteration complexities of Nesterov's fast gradient
method up to an absolute constant; \cref{app:euclidean_nesterov} describes the Euclidean
specialization.

\paragraph{Oracle complexity.} Backtracking incurs constant amortized oracle overhead per
accepted iteration, up to logarithmic terms for underestimated
initial parameters. The complexity above therefore translates into
an oracle complexity bound with these additional terms;
see \cref{sec:oracle_complexity}.

\subsection{Comparison with existing methods}\label{sec:comparison}

\paragraph{Inputs.} The rates of \cref{thm:eventual_bldf_rate} depend on \(\rho\), but ABrA-GD uses neither \(\rho\)
nor \(\Omega_k\): it finds \(M_k\) by backtracking on the exit condition, and the only problem
constant it takes as input is \(\mu\). The accelerated linear rate of \citet{chen2026accelerated}
requires the constant of their condition as input. To our knowledge, \cref{thm:eventual_bldf_rate}
gives the first accelerated linear rate for relatively smooth and relatively strongly convex
objectives obtained by a method with these inputs. \Cref{tab:method_comparison} summarizes the
comparison.

\paragraph{Triangle scaling and a posteriori bounds.} ABPG converges at rate \(O(k^{-\gamma})\) under a triangle-scaling property with exponent
\(\gamma\) and gain \(1\) \citep{hanzely2021accelerated,savchuk2024accelerated}. For standard
non-Euclidean kernels, this exponent is at most \(1\): \(\gamma=1\) for jointly convex divergences, and \(\gamma\le1/2\) for the Burg kernel (\cref{app:comparison}). This
guarantee therefore gives no acceleration, and some additional condition is necessary, since
\(O(1/k)\) is optimal for the full relatively smooth class \citep{dragomir2022optimal}. ABrA-GD
shares this limit: if \(M_k\) grows without bound, \(y_k\) approaches \(x_k^{\mathrm{best}}\), the
iterations approach Bregman gradient steps, and the safeguard of \cref{app:numexp} retains the
\(O(1/k)\) rate.

The accelerated variants instead fix \(\gamma=2\) and
backtrack on the gain, so their rate \(O(\bar G_k/k^2)\) depends on the accepted gains and holds
a posteriori, as in \cref{thm:generic_accelerated_rate_Mk}. \textbf{Bounded distortion is
implied by a uniform gain}: if triangle scaling holds with exponent \(2\) and a gain \(G\) uniform
over triples, then \(\rho_{d^*}^2\le G\) on the corresponding dual region (\cref{app:tsp}).
Hence, whenever ABPG has a uniform gain, \cref{thm:eventual_bldf_rate} also gives ABrA-GD an
accelerated rate.

\begin{table}[ht]
\centering
\small
\begin{tabular}{lcccc}
\toprule
\textbf{Method} &
\textbf{Smooth. adaptive} &
\textbf{Geometry adaptive} &
\textbf{Convex rate} &
\textbf{Acc. $\mu>0$} \\
\midrule
BPG / BPG-LS &
\checkmark\ (BPG-LS) &
-- &
$O(1/k)$ &
-- \\
ABPG &
-- &
-- &
$O(k^{-\gamma})$, global $\gamma$ &
-- \\
ABPG-e &
-- &
\checkmark\ (exponent search) &
$O(k^{-\gamma_k})$ &
-- \\
ABPG-g &
-- &
\checkmark\ (gain search) &
$O(\bar G_k/k^2)$ &
-- \\
Acc-MD (GCS) &
-- &
-- &
$O(C/k^2)$ &
\checkmark \\
ABrA-GD &
\checkmark &
\checkmark\ (via $M_k$) &
$O\!\left((\sum_i M_i^{-1/2})^{-2}\right)$ &
\checkmark \\
\bottomrule
\end{tabular}

\caption{Comparison in the relatively smooth setting. BPG-LS denotes Bregman proximal
gradient (-LS: with line search). ABPG-e uses an exponent search while ABPG-g uses a gain search. Acc-MD
\citep{chen2026accelerated} takes the constant \(C\) of its generalized Cauchy--Schwarz condition
as input.}
\label{tab:method_comparison}
\end{table}
\section{Numerical Experiments}
\label{sec:experiments}

\begin{figure}[!t]
    \centering
    \includegraphics[width=0.85\linewidth]{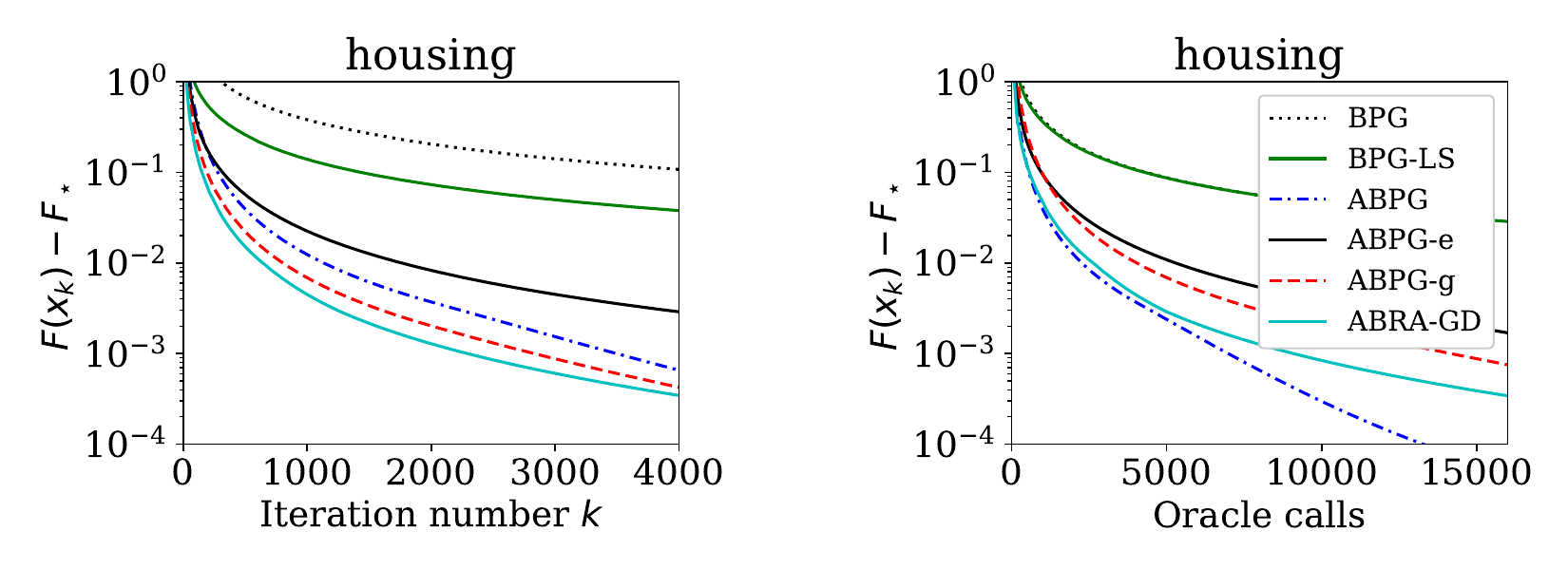}
    \includegraphics[width=0.85\linewidth]{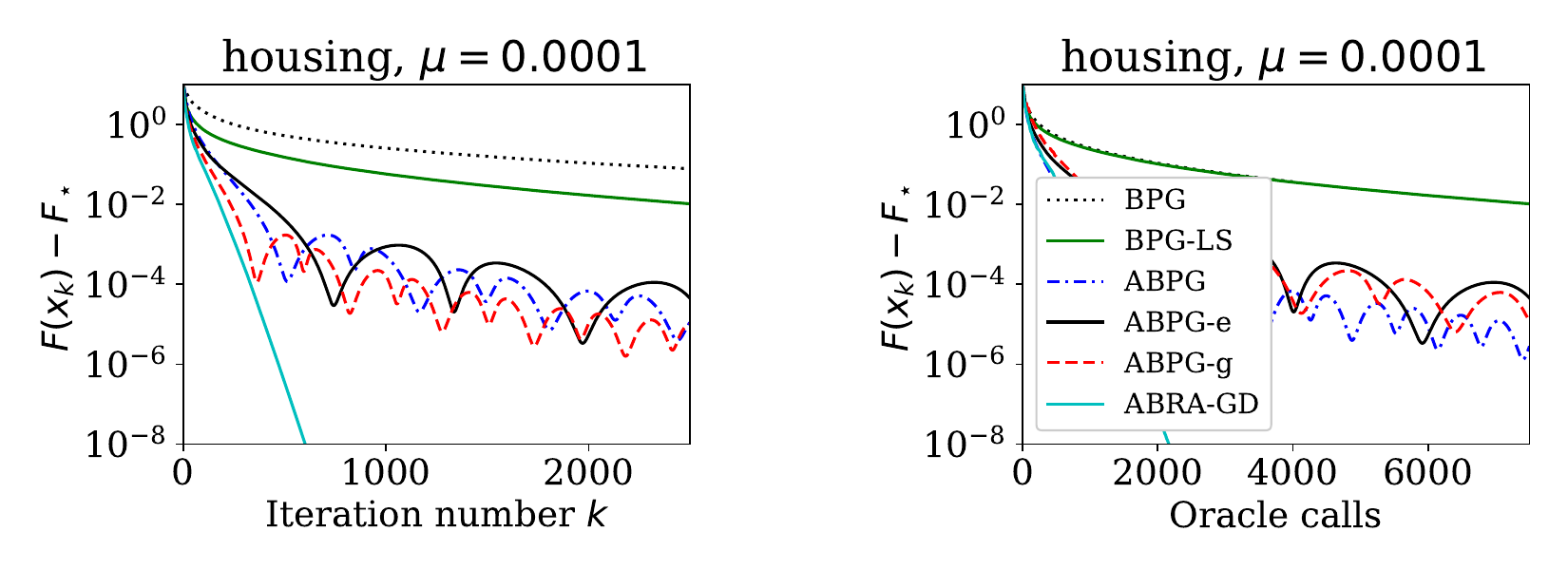}
    \caption{D-optimal design on the housing dataset. Panels show the objective residual
    \(F-F_\star\) against iterations and oracle calls. Methods are BPG, BPG with line search
    (BPG-LS), ABPG, its exponent-adaptive and gain-adaptive variants (ABPG-e and ABPG-g), and
    ABrA-GD. ABPG and ABPG-g use \(\gamma=2\). Top: convex case (\(\mu=0\)), where ABrA-GD
    performs similarly to the best ABPG variants. Bottom: relatively strongly convex case
    (\(\mu=10^{-4}\)), where ABrA-GD reaches a lower residual and its log-residual decreases
    approximately linearly. Additional datasets and accepted parameters are reported in
    \cref{fig:dopt_convex_bundle,fig:dopt_strcvx_bundle}; the \(\gamma=1\) comparison appears in
    \cref{fig:dopt_convex_bundle_gamma1}.}
    \label{fig:numex_main}
\end{figure}

We evaluate ABrA-GD on D-optimal design and Poisson linear inverse problems. The formulations,
reference functions, datasets, baselines, and initialization follow
\citet{hanzely2021accelerated}. Further protocol details are given in \cref{app:numexp_protocol},
and additional comparisons appear in \cref{app:numexp}.
Oracle calls count first-order oracle calls, including those of trials rejected by
backtracking.

For D-optimal design, we use \(\gamma=2\) in the ABPG comparisons, following the experimental
choices of \citet{hanzely2021accelerated}. With \(\gamma=1\), ABPG has only an \(O(1/k)\)
guarantee and converged more slowly than BPG in their experiments; \cref{fig:dopt_convex_bundle_gamma1}
reproduces this comparison.

\Cref{fig:numex_main} shows the housing dataset. In the convex case, ABrA-GD improves on BPG and
BPG-LS and performs similarly to the best ABPG variants over the plotted range. With
\(\mu=10^{-4}\), its log-residual decreases approximately linearly, and it reaches a lower
residual with fewer oracle calls than the displayed baselines. The other three D-optimal design
datasets show the same pattern (\cref{app:numexp}).
\Cref{fig:dopt_convex_bundle,fig:dopt_strcvx_bundle} also plot the accepted \(M_k\) values. The D-optimal design instances are scaled so that \(L=1\) \citep{hanzely2021accelerated}, and on
these instances \(L_{\mathrm{cur}}\) stays between \(1/4\) and \(2\).

On the Poisson problems, ABrA-GD improves on the nonaccelerated baselines. Its performance relative
to the accelerated methods depends on the instance and variant. We report Poisson L1 with
\((m,n)=(200,100)\) and Poisson L2 with \((m,n)=(100,1000)\) in \cref{app:numexp}.
\section{Conclusion}
\label{sec:conclusion}

ABrA-GD combines a Bregman gradient step with a dual lower model to satisfy an envelope descent
requirement. Its convergence bound depends on the acceleration parameters accepted during the run.
The dual BLDF measures the local geometric distortion in the comparison between the two steps,
but the algorithm does not need to evaluate it.

Under bounded distortion, the analysis gives an
accelerated \(O(1/k^2)\) rate for convex objectives and an accelerated linear rate for relatively
strongly convex objectives. In the D-optimal design experiments, ABrA-GD performs similarly to
adaptive ABPG in the convex case and reaches lower residuals in the relatively strongly convex
case. Performance on the Poisson problems depends on the instance.

Future directions include sharper local distortion bounds, an accelerated analysis for composite objectives (\cref{app:composite_subgradient}), and stochastic variants.

\label{endmaintext}

\clearpage
\bibliographystyle{plainnat}
\bibliography{biblio}

\begin{thebibliography}{26}
\providecommand{\natexlab}[1]{#1}
\providecommand{\url}[1]{\texttt{#1}}
\expandafter\ifx\csname urlstyle\endcsname\relax
  \providecommand{\doi}[1]{doi: #1}\else
  \providecommand{\doi}{doi: \begingroup \urlstyle{rm}\Url}\fi

\bibitem[Allen-Zhu and Orecchia(2014)]{allen2014linear}
Zeyuan Allen-Zhu and Lorenzo Orecchia.
\newblock Linear coupling: An ultimate unification of gradient and mirror
  descent.
\newblock \emph{arXiv preprint arXiv:1407.1537}, 2014.

\bibitem[Auslender and Teboulle(2006)]{auslender2006interior}
Alfred Auslender and Marc Teboulle.
\newblock Interior gradient and proximal methods for convex and conic
  optimization.
\newblock \emph{SIAM Journal on Optimization}, 16\penalty0 (3):\penalty0
  697--725, 2006.

\bibitem[Bach(2015)]{bach2015duality}
Francis Bach.
\newblock Duality between subgradient and conditional gradient methods.
\newblock \emph{SIAM Journal on Optimization}, 25\penalty0 (1):\penalty0
  115--129, 2015.

\bibitem[Bauschke et~al.(2017)Bauschke, Bolte, and
  Teboulle]{bauschke2017descent}
Heinz~H. Bauschke, J{\'e}r{\^o}me Bolte, and Marc Teboulle.
\newblock A descent lemma beyond lipschitz gradient continuity: first-order
  methods revisited and applications.
\newblock \emph{Mathematics of Operations Research}, 42\penalty0 (2):\penalty0
  330--348, 2017.

\bibitem[Bauschke et~al.(2018)Bauschke, Dao, and
  Lindstrom]{bauschke2018regularizing}
Heinz~H Bauschke, Minh~N Dao, and Scott~B Lindstrom.
\newblock Regularizing with bregman--moreau envelopes.
\newblock \emph{SIAM Journal on Optimization}, 28\penalty0 (4):\penalty0
  3208--3228, 2018.

\bibitem[Bregman(1967)]{bregman1967relaxation}
Lev~M Bregman.
\newblock The relaxation method of finding the common point of convex sets and
  its application to the solution of problems in convex programming.
\newblock \emph{USSR computational mathematics and mathematical physics},
  7\penalty0 (3):\penalty0 200--217, 1967.

\bibitem[Censor and Lent(1981)]{censor1981iterative}
Yair Censor and Arnold Lent.
\newblock An iterative row-action method for interval convex programming.
\newblock \emph{Journal of Optimization theory and Applications}, 34\penalty0
  (3):\penalty0 321--353, 1981.

\bibitem[Censor and Zenios(1992)]{censor1992proximal}
Yair Censor and Stavros~Andrea Zenios.
\newblock Proximal minimization algorithm with d-functions.
\newblock \emph{Journal of Optimization Theory and Applications}, 73\penalty0
  (3):\penalty0 451--464, 1992.

\bibitem[Chang and Lin(2011)]{chang2011libsvm}
Chih-Chung Chang and Chih-Jen Lin.
\newblock {LIBSVM}: A library for support vector machines.
\newblock \emph{ACM Transactions on Intelligent Systems and Technology},
  2\penalty0 (3):\penalty0 1--27, 2011.

\bibitem[Chen et~al.(2026)Chen, Luo, Wei, Xu, and Yao]{chen2026accelerated}
Long Chen, Hao Luo, Jingrong Wei, Zeyi Xu, and Yuan Yao.
\newblock Accelerated mirror descent method through variable and operator
  splitting.
\newblock \emph{arXiv preprint arXiv:2601.19038}, 2026.

\bibitem[Diakonikolas and Orecchia(2019)]{diakonikolas2019approximate}
Jelena Diakonikolas and Lorenzo Orecchia.
\newblock The approximate duality gap technique: A unified theory of
  first-order methods.
\newblock \emph{SIAM Journal on Optimization}, 29\penalty0 (1):\penalty0
  660--689, 2019.

\bibitem[Dragomir et~al.(2022)Dragomir, Taylor, d’Aspremont, and
  Bolte]{dragomir2022optimal}
Radu-Alexandru Dragomir, Adrien~B Taylor, Alexandre d’Aspremont, and
  J{\'e}r{\^o}me Bolte.
\newblock Optimal complexity and certification of bregman first-order methods.
\newblock \emph{Mathematical Programming}, 194\penalty0 (1):\penalty0 41--83,
  2022.

\bibitem[Hanzely et~al.(2021)Hanzely, Richtarik, and
  Xiao]{hanzely2021accelerated}
Filip Hanzely, Peter Richtarik, and Lin Xiao.
\newblock Accelerated bregman proximal gradient methods for relatively smooth
  convex optimization.
\newblock \emph{Computational Optimization and Applications}, 79\penalty0
  (2):\penalty0 405--440, 2021.

\bibitem[Karimireddy et~al.(2018)Karimireddy, Stich, and
  Jaggi]{karimireddy2018global}
Sai~Praneeth Karimireddy, Sebastian~U. Stich, and Martin Jaggi.
\newblock Global linear convergence of {N}ewton's method without
  strong-convexity or {L}ipschitz gradients.
\newblock \emph{arXiv preprint arXiv:1806.00413}, 2018.

\bibitem[Lemar{\'e}chal and Sagastiz{\'a}bal(1997)]{lemarechal1997practical}
Claude Lemar{\'e}chal and Claudia Sagastiz{\'a}bal.
\newblock Practical aspects of the moreau--yosida regularization: Theoretical
  preliminaries.
\newblock \emph{SIAM journal on optimization}, 7\penalty0 (2):\penalty0
  367--385, 1997.

\bibitem[Lin et~al.(2015)Lin, Mairal, and Harchaoui]{lin2015universal}
Hongzhou Lin, Julien Mairal, and Zaid Harchaoui.
\newblock A universal catalyst for first-order optimization.
\newblock \emph{Advances in neural information processing systems}, 28, 2015.

\bibitem[Lin et~al.(2018)Lin, Mairal, and Harchaoui]{lin2018catalyst}
Hongzhou Lin, Julien Mairal, and Zaid Harchaoui.
\newblock Catalyst acceleration for first-order convex optimization: from
  theory to practice.
\newblock \emph{Journal of Machine Learning Research}, 18\penalty0
  (212):\penalty0 1--54, 2018.

\bibitem[Liu and Liu(2022)]{liu2022dual}
Jin-Zan Liu and Xin-Wei Liu.
\newblock A dual bregman proximal gradient method for relatively-strongly
  convex optimization.
\newblock \emph{Numerical Algebra, Control and Optimization}, 12\penalty0
  (4):\penalty0 679--692, 2022.
\newblock \doi{10.3934/naco.2021028}.

\bibitem[Lu et~al.(2018)Lu, Freund, and Nesterov]{lu2018relatively}
Haihao Lu, Robert~M Freund, and Yurii Nesterov.
\newblock Relatively smooth convex optimization by first-order methods, and
  applications.
\newblock \emph{SIAM Journal on Optimization}, 28\penalty0 (1):\penalty0
  333--354, 2018.

\bibitem[Monteiro and Svaiter(2013)]{monteiro2013accelerated}
Renato~DC Monteiro and Benar~Fux Svaiter.
\newblock An accelerated hybrid proximal extragradient method for convex
  optimization and its implications to second-order methods.
\newblock \emph{SIAM Journal on Optimization}, 23\penalty0 (2):\penalty0
  1092--1125, 2013.

\bibitem[Nesterov(2004)]{nesterov2004introductory}
Yurii Nesterov.
\newblock \emph{Introductory lectures on convex optimization}.
\newblock Springer, 2004.

\bibitem[Nesterov and Polyak(2006)]{nesterov2006cubic}
Yurii Nesterov and Boris~T Polyak.
\newblock Cubic regularization of {N}ewton method and its global performance.
\newblock \emph{Mathematical Programming}, 108\penalty0 (1):\penalty0 177--205,
  2006.

\bibitem[Rockafellar(1997)]{rockafellar1997convex}
R~Tyrrell Rockafellar.
\newblock \emph{Convex analysis}, volume~28.
\newblock Princeton university press, 1997.

\bibitem[Savchuk et~al.(2024)Savchuk, Alkousa, Shushko, Vyguzov, Stonyakin,
  Pasechnyuk, and Gasnikov]{savchuk2024accelerated}
OS~Savchuk, MS~Alkousa, AS~Shushko, AA~Vyguzov, FS~Stonyakin, DA~Pasechnyuk,
  and AV~Gasnikov.
\newblock Accelerated bregman gradient methods for relatively smooth and
  relatively lipschitz continuous minimization problems.
\newblock \emph{arXiv preprint arXiv:2411.16743}, 2024.

\bibitem[Tseng(2008)]{tseng2008accelerated}
Paul Tseng.
\newblock On accelerated proximal gradient methods for convex-concave
  optimization.
\newblock \emph{submitted to SIAM Journal on Optimization}, 2\penalty0 (3),
  2008.

\bibitem[Wilson et~al.(2021)Wilson, Recht, and Jordan]{wilson2021lyapunov}
Ashia~C Wilson, Ben Recht, and Michael~I Jordan.
\newblock A lyapunov analysis of accelerated methods in optimization.
\newblock \emph{Journal of Machine Learning Research}, 22\penalty0
  (113):\penalty0 1--34, 2021.

\end{thebibliography}

\clearpage
\appendix

\tableofcontents

\section{Notations and definitions}\label{app:defs}

\begin{definition}[Legendre function]
A function $d$ is \emph{Legendre} if it is proper, closed and convex, and is both
\emph{essentially smooth} and \emph{essentially strictly convex}: it is differentiable on
$\operatorname{int}(\operatorname{dom} d)$ with
$\|\nabla d(x_k)\|\to\infty$ whenever $x_k\to\partial\operatorname{dom} d$, and it is
strictly convex on every convex subset of $\operatorname{dom}\partial d$. Essential strict
convexity is what makes $\nabla d$ injective.
In this case, $\nabla d$ is a bijection between
$\operatorname{int}(\operatorname{dom} d)$ and
$\operatorname{int}(\operatorname{dom} d^*)$ with inverse $\nabla d^*$.
\end{definition}

\begin{definition}[Bregman Divergence \citep{bregman1967relaxation}]
The \emph{Bregman divergence} of a Legendre function $d$ is
\[
D_d(x,y)
:=
d(x)-d(y)-\langle \nabla d(y),\, x-y\rangle,
\qquad x\in \operatorname{dom} d,\; y\in \operatorname{int}(\mathcal{X}).
\]
\end{definition}

\begin{definition}[Relative smoothness and strong convexity] \label{def:rel_smooth_str_cvx} We say that $f$ is $\mu$-strongly convex (convex if $\mu = 0$) and $L$-smooth relative to a Legendre function $d$ if there exists $0\leq \mu\leq L$ s.t.
\[
\mu D_d(x,y) \le f(x) - f(y) - \langle \nabla f(y),x-y\rangle \le LD_d(x,y)
\qquad \forall x,y\in\operatorname{int}(\mathcal{X}).
\]
\end{definition}

\begin{definition}[Fenchel--Legendre conjugate {\citep[\S12]{rockafellar1997convex}}]
Let $f$ be a proper function.
Its conjugate is
\[
f^*(\lambda)
:=
\sup_{x\in\mathbb{R}^n}
\{\langle \lambda,x\rangle - f(x)\}.
\]
\end{definition}

\begin{definition}[Primal and dual local norms]\label{eq:primal_dual_norm}
Let $d$ be a twice differentiable Legendre function whose Hessian is positive definite on
$\operatorname{int}(\dom d)$. For
$x \in \operatorname{int}(\dom d)$,
$w \in \operatorname{int}(\dom d^*)$,
$v \in \mathcal X$, and $s \in \mathcal X^*$, define
\[
    \|v\|_x^2 := \tfrac{1}{2}\langle \nabla^2 d(x)\,v,\, v\rangle,
    \qquad
    \|s\|_{w,*}^2 := \tfrac{1}{2}\langle s,\, \nabla^2 d^*(w)\,s\rangle.
\]
\end{definition}

Positive definiteness is a hypothesis, not a consequence of the rest: \(d(x)=x^4/4\) is Legendre and
twice differentiable on \(\mathbb R\) with \(d''(0)=0\), so the first form is degenerate at the
origin and \(d^*\) has no finite second derivative there. Under the hypothesis,
\(\nabla^2d^*(w)=\nabla^2d(\nabla d^*(w))^{-1}\) is positive definite as well and both forms are
norms. Every statement below that uses local norms or ratios of Hessians carries this hypothesis on
the region where it is applied.

The Bregman divergence admits an integral representation in terms of the local norms.
\begin{proposition}[Bregman divergence and local norms]\label{prop:bregman_local_norm}
Let $d$ be a Legendre function of class $C^2$ with positive definite Hessian on
$\operatorname{int}(\dom d)$. For any $x_1,x_2\in \operatorname{int}(\dom d)$, there is a point $z$ on the segment $[x_1,x_2]$ such that
\[
D_d(x_1,x_2)=\|x_1-x_2\|_{z}^{2}.
\]
\end{proposition}
\begin{proof}
Since $\operatorname{int}(\dom d)$ is convex, the segment $[x_1,x_2]$ lies in it, and $g(t):=d\big(x_2+t(x_1-x_2)\big)$ is twice continuously differentiable on a neighbourhood of $[0,1]$. Taylor's theorem with the Lagrange remainder gives a $\theta\in(0,1)$ with $g(1)=g(0)+g'(0)+\tfrac{1}{2}g''(\theta)$, that is, writing $z:=x_2+\theta(x_1-x_2)$,
\[
d(x_1)-d(x_2)-\langle \nabla d(x_2),\,x_1-x_2\rangle
=
\tfrac{1}{2}\big\langle \nabla^2 d(z)(x_1-x_2),\,x_1-x_2\big\rangle.
\]
The left-hand side is $D_d(x_1,x_2)$ and the right-hand side is $\|x_1-x_2\|_z^2$ by \cref{eq:primal_dual_norm}.
\end{proof}

\subsection{The anchored decomposition}\label{app:anchored_decomposition}

For a differentiable \(g\), write
\[
D_g(x,y):=g(x)-g(y)-\langle\nabla g(y),x-y\rangle ,
\]
which is the Bregman divergence already defined for \(d\). In this notation, \(g\) is \(L\)-smooth
relative to \(d\) when \(D_g\le LD_d\) and \(\mu\)-strongly convex relative to \(d\) when
\(D_g\ge\mu D_d\), both on \(\operatorname{int}(\dom d)\).

The anchor \(x\mapsto\mu D_d(x,x_0)\) differs from \(\mu d\) by an affine function, and an affine
function has zero Bregman divergence. So the anchor contributes exactly \(\mu D_d\), and
\(\phi=f+\mu D_d(\cdot,x_0)\) satisfies
\begin{equation}\label{eq:anchor_divergence_split}
D_\phi(x,y)=D_f(x,y)+\mu D_d(x,y),
\qquad x,y\in\operatorname{int}(\dom d).
\end{equation}

Let \(\phi\) be \(L\)-smooth and \(\mu\)-strongly convex relative to \(d\), that is
\(\mu D_d\le D_\phi\le LD_d\). Strong convexity says precisely that \(\phi-\mu D_d(\cdot,x_0)\) is
convex, so put
\[
f:=\phi-\mu D_d(\cdot,x_0).
\]
Subtracting \(\mu D_d\) throughout \cref{eq:anchor_divergence_split} gives
\(0\le D_f\le(L-\mu)D_d\): the function \(f\) is convex and \((L-\mu)\)-smooth relative to \(d\),
which is \cref{ass:setting}, and \(\phi\) is back in the form \cref{eq:phi_objective_boxed}.

Conversely, if \(f\) satisfies \cref{ass:setting} then \(0\le D_f\le(L-\mu)D_d\), and
\cref{eq:anchor_divergence_split} returns \(\mu D_d\le D_\phi\le LD_d\).

The two hypotheses are therefore the same one. All of the relative strong convexity of \(\phi\) sits
in the anchor term, and the anchor is what raises the relative smoothness constant from \(L-\mu\)
to \(L\).

\subsection{Affine equality constraints}\label{app:affine_restriction}

The analysis also applies when the feasible set includes affine equality
constraints. Let \(\mathcal A\subseteq\mathbb R^n\) be an affine subspace,
assume that \(x_0\in\mathcal A\cap\operatorname{int}(\dom d)\), and consider
the feasible set
\[
    \mathcal X=\operatorname{cl}(\dom d)\cap\mathcal A.
\]
We can eliminate the equality constraints by choosing coordinates on
\(\mathcal A\).

Let the columns of \(U\) form an orthonormal basis for the direction
space of \(\mathcal A\). Every \(x\in\mathcal A\) has a unique
representation \(x=x_0+Uu\). Define
\[
    \widetilde f(u)=f(x_0+Uu),
    \qquad
    \widetilde d(u)=d(x_0+Uu).
\]
The Bregman divergence is preserved under this change of variables:
\[
    D_{\widetilde d}(u,v)
    =D_d(x_0+Uu,x_0+Uv).
\]
The constrained problem therefore becomes
\[
    \min_{u\in\operatorname{cl}(\dom\widetilde d)}
    \widetilde f(u)+\mu D_{\widetilde d}(u,0).
\]
The relative smoothness and strong convexity inequalities carry over
directly, with the same constants \(L\) and \(\mu\).

The restricted kernel \(\widetilde d\) is Legendre. To see this, write
\(V=\operatorname{range}(U)\). Since \(\mathcal A\) intersects
\(\operatorname{int}(\dom d)\), the subdifferential sum rule gives
\[
    \partial(d+\delta_{\mathcal A})(x)
    =\partial d(x)+V^\perp
\]
\citep[Thm.~23.8]{rockafellar1997convex}. At a boundary point of the
restricted domain, relative to \(\mathcal A\), the right-hand side is
empty because \(d\) is Legendre. This gives essential smoothness in the
reduced coordinates. Essential strict convexity is inherited from \(d\).

Thus the problem in the variable \(u\) satisfies
\cref{ass:setting}. All updates and proofs apply in these coordinates,
so affine equality constraints require no change to the analysis.
\clearpage
\section{Recovering Bregman Gradient Descent and Dual Averaging}
\label{app:envelope_special_cases}

The implicit envelope of \cref{sec:envelope} specializes to the two classical methods. Both
derivations are given here: they motivate the accelerated step of \cref{algo:bregman_adaptive}
but are not needed to read it.

\subsection{Mirror Gradient Descent}

Since $f$ is relatively smooth, for all $x_{k}$, we have the following upper model $\mathcal{U}(x,x_k)$ for $f$,
\[
    f(x) \leq \mathcal{U}(x,x_k) := f(x_k) + \langle \nabla f(x_k),x-x_k\rangle + (L-\mu)D_d(x,x_k).
\]
By definition of $\phi$, it is sufficient to find an iterate $x_{k+1}$ and a stepsize $\eta_{k+1}$ such that 
\[
    \mathcal{U}(x_{k+1},x_{k}) + \mu D_d(x_{k+1},x_0) \leq \phi_\star^{\eta_{k+1}}
\]
Naturally, we can take the minimum of the upper bound and adjust $\eta_{k+1}$ accordingly,
\begin{equation*}
    x_{k+1} = \argmin_x \mathcal{U}(x,x_{k}) + \mu D_d(x,x_0) = \nabla d^*\Big( \nabla d(x_k) - \frac{1}{L}\nabla \phi(x_k) \Big). \label{eq:gradient_descent}
\end{equation*}
This is exactly gradient descent, and it fits \cref{algo:implicit_descent} with a certain schedule for $\eta_k$.

\begin{proposition}\label{prop:rate_grad}
The Bregman gradient descent iterates satisfy \cref{algo:implicit_descent} with
\[
\eta_{k}^{-1}
=
L\min\left\{
\frac{1}{k},
\frac{\kappa}{(1+\kappa)^{k}-1}
\right\},
\qquad
\kappa=\frac{\mu}{L},
\]
where the second term is omitted when \(\mu=0\).
\end{proposition}

The two branches are the known rates of the Bregman gradient method: the \(O(1/k)\) rate under
relative smoothness \citep{bauschke2017descent,lu2018relatively}, and the linear rate under relative
strong convexity \citep{lu2018relatively}. What the proposition adds is the translation into the
envelope schedule, that is, the \(\eta_k\) for which the method satisfies the requirement of
\cref{algo:implicit_descent}.

\subsection{Saddle formulation of the envelope}
\label{app:saddle_formulation}

Define
\begin{equation}\label{eq:saddle_formulation}
S^\eta(x,\lambda)
:=
\langle x,\lambda\rangle-f^*(\lambda)
+\left(\mu+\frac1\eta\right)D_d(x,x_0).
\end{equation}
The Fenchel representation of \(f\) and the definition of
\(\mathcal D^\eta\) give
\[
\phi^\eta(x)=\sup_\lambda S^\eta(x,\lambda),
\qquad
\mathcal D^\eta(\lambda)=\min_x S^\eta(x,\lambda).
\]
Weak duality therefore yields
\[
\sup_\lambda\mathcal D^\eta(\lambda)
\le \min_x\sup_\lambda S^\eta(x,\lambda)
=\phi_\star^\eta.
\]
For \(\lambda=\nabla f(y)\), Fenchel equality gives
\[
S^\eta(x,\nabla f(y))
=
f(y)+\langle\nabla f(y),x-y\rangle
+\left(\mu+\frac1\eta\right)D_d(x,x_0),
\]
With \(\eta^{-1}=0\), this is the second term of the lower model \(\mathcal L_{k+1}\) of
\cref{eq:abra_lower_model}, at \(y=y_k\).

This formulation is helpful for proving the convergence rate of Dual Averaging, and therefore, AbrA-GD.

\subsection{Dual Averaging} \label{sec:dual_avg}

Even when \(S^\eta(x,\lambda)\) cannot be evaluated explicitly, we can still update the dual variable through the best-response (\textbf{BR}) maps
\begin{align}
    z^\eta(\lambda)
    &\in
    \arg\min_x S^\eta(x,\lambda)
    =
    \nabla d^*\!\left(\nabla d(x_0)-\frac{1}{\mu+\eta^{-1}}\lambda\right),
    \tag{Primal BR}
    \label{eq:primal_best_response}
    \\
    \lambda(x)
    &\in
    \arg\max_\lambda S^\eta(x,\lambda)
    =
    \nabla f(x).
    \tag{Dual BR}
    \label{eq:dual_best_response}
\end{align}
The dual-averaging update is the damped best-response-to-best-response iteration
\begin{align}
    \lambda_{k+1}
    &=
    (1-t_k)\lambda_k+t_k\lambda(z_k)
    =
    (1-t_k)\lambda_k+t_k\nabla f(z_k),
    \notag\\
    z_{k+1}
    &=z^{\eta_{k+1}}(\lambda_{k+1}).
    \tag{Dual Avg.}
    \label{eq:dual_avg}
\end{align}
This can be viewed as a damped alternating \(\min\)-\(\max\) scheme. Its main advantage is that it avoids computing \(\nabla f^*(\lambda)\), which is usually intractable.

\begin{theorem}
\label{thm:dual_averaging_implicit}
The update \cref{eq:dual_avg}, initialized with \(\eta_0=0\), \(\lambda_0=0\) and
\(z_0=x_0\), and using
\[
    t_k=\frac{\mu+\eta_k^{-1}}{L+\eta_k^{-1}},
    \qquad
    \eta_{k+1}^{-1}=(1-t_k)\eta_k^{-1},
\]
satisfies \cref{algo:implicit_descent} with
\[
    x_{k}\in\argmin_{1\le i\le k}\phi(z_i),\quad \eta_k^{-1}
    \le
    L\min\left\{
    \frac1k,
    \frac{\kappa}{(1+\kappa)^k-1}
    \right\},
    \qquad
    \kappa=\frac{\mu}{L},
\]
where the second term is omitted when \(\mu=0\).
\end{theorem}
\section{Technical Lemmas}
\label{app:technical_lemmas}

\begin{lemma}[Three-point property of Tseng \citep{tseng2008accelerated}]
\label{lem:three_point}
Let $g:\mathcal X\to\mathbb R\cup\{+\infty\}$ be a convex function, and let $D_d(\cdot,\cdot)$ be the Bregman divergence induced by $d$. For a given $z\in\mathcal X$, define
\[
z_+ \;:=\;
\arg\min_{x\in\mathcal X}
\big\{
g(x) + D_d(x,z)
\big\}.
\]
Then, for all $x\in\mathcal X$,
\[
g(x) + D_d(x,z)
\;\ge\;
g(z_+) + D_d(z_+,z) + D_d(x,z_+).
\]
\end{lemma}

Note that when the minimization is unconstrained and \(g\) is affine, the inequality holds with equality,
\begin{equation}
g(x) + D_d(x,z)
\;=\;
g(z_+) + D_d(z_+,z) + D_d(x,z_+). \label{eq:three_point_exact}
\end{equation}

\begin{lemma}[Growth around the minimizer]
\label{lem:rel_strong_min}
Let $g$ be differentiable and $\alpha$-strongly convex relative to $d$, and let
$z:=\arg\min_x g(x)$ lie in the interior of the domain of $d$. Then, for all $x$,
\[
g(x)\;\ge\;g(z)+\alpha D_d(x,z).
\]
If $g-\alpha d$ is affine, the inequality holds with equality.
\end{lemma}
\begin{proof}
The function $h:=g-\alpha d$ is convex, so $h(x)\ge h(z)+\langle\nabla h(z),x-z\rangle$, with equality
when $h$ is affine. Since $z$ is an interior minimizer of $g$, $\nabla g(z)=0$ and
$\nabla h(z)=-\alpha\nabla d(z)$. Adding $\alpha d(x)=\alpha d(z)+\alpha\langle\nabla d(z),x-z\rangle+\alpha D_d(x,z)$
gives the result.
\end{proof}

The following generic lemma is important in the proof of Dual Averaging and Accelerated Gradient Descent.

\begin{restatable}[Dual lower-bound recursion]{lemma}{PrimalDualLowerStep}
\label{lem:pd_lower_step}
Under \cref{ass:setting}, so that $\phi$ is $L$-smooth and $\mu$-strongly convex relative to $d$, assume that, for some $b_k$,
\[
b_k\le \mathcal D^{\eta_k}(\lambda_k),
\qquad
\lambda_{k+1}=(1-t_k)\lambda_k+t_k\nabla f(y_k),
\qquad
\eta_{k+1}^{-1}=(1-t_k)\eta_k^{-1},
\]
and let
\[
z_k\in\arg\min_x S^{\eta_k}(x,\lambda_k),
\qquad
z_{k+1}\in\arg\min_x S^{\eta_{k+1}}(x,\lambda_{k+1}),
\]
where $S^\eta$ is the saddle point problem defined in \cref{eq:saddle_formulation}. Define
\[
\phi_k^{\mathrm{low}}
:=
\phi(y_k)
+
\langle\nabla\phi(y_k),z_k-y_k\rangle
+
\mu D_d(z_k,y_k).
\]
Then
\[
\mathcal D^{\eta_{k+1}}(\lambda_{k+1})
\ge
(1-t_k)b_k
+
t_k\phi_k^{\mathrm{low}}
-
\left(\mu+\eta_{k+1}^{-1}\right)D_d(z_k,z_{k+1}).
\]
If \(t_k=1\), the same conclusion holds with the term \((1-t_k)b_k\) omitted, for any
\(\eta_{k+1}>0\), and without the hypothesis \(b_k\le\mathcal D^{\eta_k}(\lambda_k)\) or the
recursion \(\eta_{k+1}^{-1}=(1-t_k)\eta_k^{-1}\):
\[
\mathcal D^{\eta_{k+1}}(\lambda_{k+1})
\ge
\phi_k^{\mathrm{low}}
-
\left(\mu+\eta_{k+1}^{-1}\right)D_d(z_k,z_{k+1}).
\]
\end{restatable}
\begin{proof}
By concavity of \(S^\eta(x,\cdot)\), evaluated at \(x=z_{k+1}\),
\[
\mathcal D^{\eta_{k+1}}(\lambda_{k+1})
=
S^{\eta_{k+1}}(z_{k+1},\lambda_{k+1})
\ge
(1-t_k)S^{\eta_{k+1}}(z_{k+1},\lambda_k)
+
t_kS^{\eta_{k+1}}(z_{k+1},\nabla f(y_k)).
\]
Using
\[
S^{\eta_{k+1}}(x,\lambda_k)
=
S^{\eta_k}(x,\lambda_k)
+
(\eta_{k+1}^{-1}-\eta_k^{-1})D_d(x,x_0),
\]
\[
S^{\eta_k}(x,\lambda_k)
\ge
b_k+(\mu+\eta_k^{-1})D_d(x,z_k),
\]
where the last inequality follows directly from the optimality of
\(z_k\). Indeed, with \(\alpha_k:=\mu+\eta_k^{-1}\),
\[
\begin{aligned}
S^{\eta_k}(x,\lambda_k)-S^{\eta_k}(z_k,\lambda_k)
&=
\left\langle
\lambda_k+\alpha_k\bigl(\nabla d(z_k)-\nabla d(x_0)\bigr),
 x-z_k
\right\rangle
+
\alpha_kD_d(x,z_k)
\\
&\ge \alpha_kD_d(x,z_k),
\end{aligned}
\]
and \(S^{\eta_k}(z_k,\lambda_k)=\mathcal D^{\eta_k}(\lambda_k)\ge b_k\). This is the only use of the
hypothesis on \(b_k\), and it enters the conclusion through the factor \((1-t_k)\).
Finally, since \(\eta_{k+1}^{-1}=(1-t_k)\eta_k^{-1}\), the coefficients of \(D_d(x,x_0)\) sum to
\[
(1-t_k)(\eta_{k+1}^{-1}-\eta_k^{-1})+t_k(\mu+\eta_{k+1}^{-1})=t_k\mu .
\]
The identity \(D_d(x,x_0)=D_d(y_k,x_0)+\langle\nabla d(y_k)-\nabla d(x_0),x-y_k\rangle+D_d(x,y_k)\)
turns \(t_k\big(f(y_k)+\langle\nabla f(y_k),x-y_k\rangle+\mu D_d(x,x_0)\big)\) into
\(t_k\big(\phi(y_k)+\langle\nabla\phi(y_k),x-y_k\rangle+\mu D_d(x,y_k)\big)\). Hence
\[
\mathcal D^{\eta_{k+1}}(\lambda_{k+1})
\ge
(1-t_k)b_k
+
t_k\phi(y_k)
+
H_k(z_{k+1}),
\]
where
\[
H_k(x)
:=
(1-t_k)(\mu+\eta_k^{-1})D_d(x,z_k)
+
t_k\mu D_d(x,y_k)
+
t_k\langle\nabla\phi(y_k),x-y_k\rangle .
\]

We claim that \(z_{k+1}\in\arg\min_x H_k(x)\). Indeed, using
\[
\lambda_k+(\mu+\eta_k^{-1})(\nabla d(z_k)-\nabla d(x_0))=0,
\]
\[
\lambda_{k+1}=(1-t_k)\lambda_k+t_k\nabla f(y_k),
\qquad
\nabla f(y_k)=\nabla\phi(y_k)-\mu(\nabla d(y_k)-\nabla d(x_0)),
\]
the optimality condition of
\[
z_{k+1}\in\arg\min_x S^{\eta_{k+1}}(x,\lambda_{k+1})
\]
is equivalent to
\[
\begin{aligned}
0
=(1-t_k)(\mu+\eta_k^{-1})
\bigl(\nabla d(z_{k+1})-\nabla d(z_k)\bigr)
+
t_k\mu
\bigl(\nabla d(z_{k+1})-\nabla d(y_k)\bigr)
+
t_k\nabla\phi(y_k),
\end{aligned}
\]
which is exactly the optimality condition for \(z_{k+1}\in\arg\min_x H_k(x)\).

Moreover, \(H_k\) has Bregman curvature
\[
(1-t_k)(\mu+\eta_k^{-1})+t_k\mu
=
\mu+\eta_{k+1}^{-1}.
\]
Therefore \(H_k\) is an affine function plus a weighted sum of Bregman divergences in its first
argument, with total weight \(\mu+\eta_{k+1}^{-1}\), and \(z_{k+1}\) is its unconstrained minimizer.
\Cref{eq:three_point_exact} applies and gives, at \(x=z_k\),
\[
H_k(z_{k+1})
=
H_k(z_k)
-
(\mu+\eta_{k+1}^{-1})D_d(z_k,z_{k+1}).
\]
The three-point property of \cref{lem:three_point} would give this inequality in the opposite
direction; here it holds as an equality because the minimization is unconstrained.
Since
\[
H_k(z_k)
=
t_k\mu D_d(z_k,y_k)
+
t_k\langle\nabla\phi(y_k),z_k-y_k\rangle,
\]
we obtain
\[
\begin{aligned}
\mathcal D^{\eta_{k+1}}(\lambda_{k+1})
\ge\;&
(1-t_k)b_k
+
t_k\phi(y_k)
+
t_k\langle\nabla\phi(y_k),z_k-y_k\rangle
\\
&+
t_k\mu D_d(z_k,y_k)
-
(\mu+\eta_{k+1}^{-1})D_d(z_k,z_{k+1}).
\end{aligned}
\]
By definition,
\[
\phi_k^{\mathrm{low}}
=
\phi(y_k)
+
\langle\nabla\phi(y_k),z_k-y_k\rangle
+
\mu D_d(z_k,y_k),
\]
and therefore
\[
\mathcal D^{\eta_{k+1}}(\lambda_{k+1})
\ge
(1-t_k)b_k
+
t_k\phi_k^{\mathrm{low}}
-
(\mu+\eta_{k+1}^{-1})D_d(z_k,z_{k+1}).
\]

For the case \(t_k=1\), argue directly. Then \(\lambda_{k+1}=\nabla f(y_k)\), so the concavity step
above is unnecessary and neither \(\lambda_k\) nor \(\eta_k\) appears. Writing
\(\alpha_{k+1}=\mu+\eta_{k+1}^{-1}\) and using that the conjugate is attained at \(y_k\),
\[
\mathcal D^{\eta_{k+1}}(\lambda_{k+1})
=
S^{\eta_{k+1}}(z_{k+1},\nabla f(y_k))
=
f(y_k)+\langle\nabla f(y_k),z_{k+1}-y_k\rangle+\alpha_{k+1}D_d(z_{k+1},x_0).
\]
Substituting \(\nabla f(y_k)=\nabla\phi(y_k)-\mu(\nabla d(y_k)-\nabla d(x_0))\) and
\(f(y_k)=\phi(y_k)-\mu D_d(y_k,x_0)\), and expanding \(\mu D_d(z_{k+1},x_0)\) by the three-point
identity around \(y_k\), the terms anchored at \(x_0\) combine into \(\mu D_d(z_{k+1},y_k)\):
\[
\mathcal D^{\eta_{k+1}}(\lambda_{k+1})
=
\phi(y_k)+\widetilde H_k(z_{k+1}),
\qquad
\widetilde H_k(x)
:=
\eta_{k+1}^{-1}D_d(x,x_0)
+
\mu D_d(x,y_k)
+
\langle\nabla\phi(y_k),x-y_k\rangle .
\]
The optimality condition for \(z_{k+1}\in\arg\min_xS^{\eta_{k+1}}(x,\lambda_{k+1})\) is exactly the
one for \(z_{k+1}\in\arg\min_x\widetilde H_k(x)\), and \(\widetilde H_k\) is affine plus Bregman
divergences of total weight \(\eta_{k+1}^{-1}+\mu=\alpha_{k+1}\), so \cref{eq:three_point_exact} at
\(x=z_k\) gives \(\widetilde H_k(z_{k+1})=\widetilde H_k(z_k)-\alpha_{k+1}D_d(z_k,z_{k+1})\). Since
\(\phi(y_k)+\widetilde H_k(z_k)=\phi_k^{\mathrm{low}}+\eta_{k+1}^{-1}D_d(z_k,x_0)\) and
\(D_d\ge0\),
\[
\mathcal D^{\eta_{k+1}}(\lambda_{k+1})
=
\phi_k^{\mathrm{low}}+\eta_{k+1}^{-1}D_d(z_k,x_0)-\alpha_{k+1}D_d(z_k,z_{k+1})
\ge
\phi_k^{\mathrm{low}}-\alpha_{k+1}D_d(z_k,z_{k+1}).
\]
At \(k=0\) the initialization gives \(z_0=x_0\), so the middle term is zero and the bound holds with
equality.
\end{proof}

\clearpage
\subsection{Convergence rate of gradient descent and dual averaging}
\label{app:rates}

\begin{proof}[Proof of \cref{thm:implicit_descent_rate}]
By construction, \cref{algo:implicit_descent} produces $x_k$ satisfying $\phi(x_k)\le\phi_\star^{\eta_k}$. Bounding the minimum in \cref{eq:enveloppe} by its value at $x^\star$,
\[
\phi_\star^{\eta_k}
=
\min_x\left\{\phi(x)+\frac{1}{\eta_k}D_d(x,x_0)\right\}
\;\leq\;
\phi(x^\star)+\frac{1}{\eta_k}D_d(x^\star,x_0)
=
\phi_\star+\frac{D_d(x_\star,x_0)}{\eta_k}.
\]
Chaining the two inequalities gives \cref{eq:rate_implicit}.
\end{proof}

\begin{proof}[Proof of \cref{prop:rate_grad}]
The optimality condition of the Bregman gradient step is
\[
\nabla \phi(x_k)+L\big(\nabla d(x_{k+1})-\nabla d(x_k)\big)=0.
\]
The objective of this step is \(\langle\nabla f(x_k),x\rangle+LD_d(x,w_k)\) up to a constant, with
\(\nabla d(w_k)=(1-\kappa)\nabla d(x_k)+\kappa\nabla d(x_0)\) and \(\kappa=\mu/L\).
By relative \((L-\mu)\)-smoothness of \(f\), convexity of \(f\), and \cref{eq:three_point_exact}, we obtain
\begin{equation}\label{eq:gd_one_step_implicit}
\phi(x_{k+1})
\le
f(x)+\mu D_d(x,x_0)
+(L-\mu)D_d(x,x_k)
-
L D_d(x,x_{k+1}).
\end{equation}

Let
\[
r:=\frac{L}{L-\mu}=\frac{1}{1-\kappa}.
\]
Multiplying \cref{eq:gd_one_step_implicit} by \(r^{k+1}\) and using
\[
(L-\mu)r^{k+1}=Lr^k,
\]
gives
\[
r^{k+1}\phi(x_{k+1})
\le
r^{k+1}\bigl(f(x)+\mu D_d(x,x_0)\bigr)
+
L\bigl(r^kD_d(x,x_k)-r^{k+1}D_d(x,x_{k+1})\bigr).
\]
Summing from \(i=0\) to \(k\) yields
\[
\sum_{i=0}^k r^{i+1}\phi(x_{i+1})
\le
\left(\sum_{i=0}^k r^{i+1}\right)
\bigl(f(x)+\mu D_d(x,x_0)\bigr)
+
L\bigl(D_d(x,x_0)-r^{k+1}D_d(x,x_{k+1})\bigr).
\]
Dropping the last negative term and using monotonicity of \(\phi(x_i)\), we get
\[
\phi(x_{k+1})
\le
f(x)
+
\left(
\mu+
\frac{L}{\sum_{i=0}^k r^{i+1}}
\right)D_d(x,x_0).
\]
Therefore \(x_{k+1}\) satisfies \cref{algo:implicit_descent} with
\[
\eta_{k+1}^{-1}
=
\frac{L}{\sum_{i=0}^k r^{i+1}}.
\]
If \(\mu=0\), then \(r=1\), so
\[
\eta_{k+1}^{-1}
=
\frac{L}{k+1}.
\]
If \(0<\mu<L\), then
\[
\eta_{k+1}^{-1}
=
L\frac{\kappa}{r^{k+1}-1}
=
L\frac{\kappa(1-\kappa)^{k+1}}{1-(1-\kappa)^{k+1}}.
\]
Since \(r=(1-\kappa)^{-1}\ge 1+\kappa\), we have
\[
\eta_{k+1}^{-1}
\le
L\frac{\kappa}{(1+\kappa)^{k+1}-1}.
\]
Thus the simpler valid rate is
\[
\eta_{k+1}^{-1}
\le
L\min\left\{
\frac{1}{k+1},
\frac{\kappa}{(1+\kappa)^{k+1}-1}
\right\}.
\]
\end{proof}

\begin{restatable}[Dual-averaging one-step recursion]{lemma}{DualAveragingOneStep}
\label{lem:dual_averaging_one_step}
Let
\[
\mathcal D_k:=\mathcal D^{\eta_k}(\lambda_k),
\qquad
z_k=z^{\eta_k}(\lambda_k),
\]
and assume
\[
\lambda_{k+1}
=
(1-t_k)\lambda_k+t_k\nabla f(z_k),
\qquad
z_{k+1}=z^{\eta_{k+1}}(\lambda_{k+1}).
\]
Choose
\[
t_k=\frac{\mu+\eta_k^{-1}}{L+\eta_k^{-1}},
\qquad
\eta_{k+1}^{-1}=(1-t_k)\eta_k^{-1}.
\]
Then
\[
\mathcal D_{k+1}
\ge
(1-t_k)\mathcal D_k+t_k\phi(z_{k+1}).
\]
Moreover, if \(t_k\in[0,1]\), \(x_{k+1}\in\argmin\{\phi(x_k),\phi(z_{k+1})\}\), and
\[
\phi(x_k)\le \mathcal D_k,
\]
then
\[
\phi(x_{k+1})\le \mathcal D_{k+1}.
\]
\end{restatable}

\begin{proof}
By concavity of \(S^\eta(x,\cdot)\), for every \(x\),
\[
S^{\eta_{k+1}}(x,\lambda_{k+1})
\ge
(1-t_k)S^{\eta_{k+1}}(x,\lambda_k)
+
t_kS^{\eta_{k+1}}(x,\nabla f(z_k)).
\]
Using
\[
S^{\eta_{k+1}}(x,\lambda_k)
=
S^{\eta_k}(x,\lambda_k)
+
(\eta_{k+1}^{-1}-\eta_k^{-1})D_d(x,x_0),
\]
and \(z_k=z^{\eta_k}(\lambda_k)\), we have
\[
S^{\eta_k}(x,\lambda_k)
\ge
\mathcal D_k
+
(\mu+\eta_k^{-1})D_d(x,z_k).
\]
By relative \((L-\mu)\)-smoothness of \(f\),
\[
f(z_k)+\langle\nabla f(z_k),x-z_k\rangle
\ge
f(x)-(L-\mu)D_d(x,z_k),
\]
hence
\[
S^{\eta_{k+1}}(x,\nabla f(z_k))
\ge
\phi(x)
+
\eta_{k+1}^{-1}D_d(x,x_0)
-
(L-\mu)D_d(x,z_k).
\]
Combining gives
\[
\begin{aligned}
S^{\eta_{k+1}}(x,\lambda_{k+1})
\ge\;&
(1-t_k)\mathcal D_k
+
t_k\phi(x)
\\
&+
\Big[(1-t_k)(\mu+\eta_k^{-1})-t_k(L-\mu)\Big]D_d(x,z_k)
\\
&+
\Big[\eta_{k+1}^{-1}-(1-t_k)\eta_k^{-1}\Big]D_d(x,x_0).
\end{aligned}
\]
With the chosen \(t_k\) and \(\eta_{k+1}\), both Bregman coefficients vanish, so
\[
S^{\eta_{k+1}}(x,\lambda_{k+1})
\ge
(1-t_k)\mathcal D_k+t_k\phi(x).
\]
Taking \(x=z_{k+1}=z^{\eta_{k+1}}(\lambda_{k+1})\) gives
\[
\mathcal D_{k+1}
\ge
(1-t_k)\mathcal D_k+t_k\phi(z_{k+1}).
\]

Now assume \(\phi(x_k)\le\mathcal D_k\). Since
\[
x_{k+1}\in\argmin\{\phi(x_k),\phi(z_{k+1})\},
\]
and \(t_k\in[0,1]\),
\[
\phi(x_{k+1})
\le
(1-t_k)\phi(x_k)+t_k\phi(z_{k+1})
\le
(1-t_k)\mathcal D_k+t_k\phi(z_{k+1})
\le
\mathcal D_{k+1}.
\]
\end{proof}

\begin{proof}[Proof of \cref{thm:dual_averaging_implicit}]
Let
\[
\mathcal D_k:=\mathcal D^{\eta_k}(\lambda_k).
\]
The updates in the statement are the hypotheses of \cref{lem:dual_averaging_one_step}, whose dual step is taken at \(z_k\). Its remaining hypothesis \(z_k=z^{\eta_k}(\lambda_k)\) holds for \(k\ge1\) because that is how \cref{eq:dual_avg} produced \(z_k\), and at \(k=0\) because \(\eta_0=0\) and \(\lambda_0=0\) give \(z^{\eta_0}(\lambda_0)=x_0=z_0\). It gives
\[
\mathcal D_{k+1}
\ge
(1-t_k)\mathcal D_k+t_k\phi(z_{k+1}).
\]

With the limiting initialization \(\eta_0=0\), we have \(t_0=1\), so
\[
\mathcal D_1\ge\phi(z_1).
\]
Since \(x_1\in\argmin_{1\le i\le1}\phi(z_i)\), this gives
\[
\phi(x_1)\le\mathcal D_1.
\]
Assume now that
\[
\phi(x_k)\le\mathcal D_k.
\]
Because
\[
x_{k+1}\in\argmin_{1\le i\le k+1}\phi(z_i),
\]
we have
\[
\phi(x_{k+1})
\le
(1-t_k)\phi(x_k)+t_k\phi(z_{k+1}).
\]
Thus
\[
\phi(x_{k+1})
\le
(1-t_k)\mathcal D_k+t_k\phi(z_{k+1})
\le
\mathcal D_{k+1}.
\]
By induction,
\[
\phi(x_k)\le\mathcal D_k
\qquad \forall k\ge1.
\]
Finally,
\[
\mathcal D_k
=
\min_x S^{\eta_k}(x,\lambda_k)
\le
\phi_\star^{\eta_k},
\]
so
\[
\phi(x_k)\le\phi_\star^{\eta_k}.
\]
Thus the update satisfies \cref{algo:implicit_descent}.

It remains to estimate \(\eta_k^{-1}\). If \(\mu=0\), then
\[
t_k=\frac{\eta_k^{-1}}{L+\eta_k^{-1}},
\qquad
\eta_{k+1}^{-1}
=
\frac{L\eta_k^{-1}}{L+\eta_k^{-1}}.
\]
Equivalently,
\[
\eta_{k+1}=\eta_k+\frac1L.
\]
With \(\eta_0=0\), this gives
\[
\eta_k^{-1}=\frac{L}{k}.
\]

If \(\mu>0\), then, with \(\kappa:=\mu/L\),
\[
\eta_{k+1}^{-1}
=
(1-t_k)\eta_k^{-1}
=
\frac{(L-\mu)\eta_k^{-1}}{L+\eta_k^{-1}}.
\]
Hence
\[
\frac{L}{\eta_{k+1}^{-1}}
=
\frac{1}{1-\kappa}
\left(
\frac{L}{\eta_k^{-1}}+1
\right).
\]
With the limiting initialization \(\eta_0^{-1}=+\infty\), this yields
\[
\frac{L}{\eta_k^{-1}}
=
\sum_{i=1}^k(1-\kappa)^{-i}
=
\frac{(1-\kappa)^{-k}-1}{\kappa}.
\]
Thus
\[
\eta_k^{-1}
=
L\frac{\kappa}{(1-\kappa)^{-k}-1}
\le
L\frac{\kappa}{(1+\kappa)^k-1},
\]
where we used
\[
(1-\kappa)^{-1}\ge1+\kappa.
\]
Each term of the sum \(\sum_{i=1}^k(1-\kappa)^{-i}\) above is at least \(1\), so \(L/\eta_k^{-1}\ge k\), that is,
\[
\eta_k^{-1}\le\frac{L}{k}.
\]
Therefore
\[
\eta_k^{-1}
\le
L\min\left\{
\frac1k,
\frac{\kappa}{(1+\kappa)^k-1}
\right\}.
\]
\end{proof}

\clearpage
\subsection{Technical results on BLDF}

\begin{table}[t]
\centering
\renewcommand{\arraystretch}{1.15}
\begin{tabular}{lccc}
\toprule
\textbf{Geometry}
& $d_i(x)$
& \textbf{Dual form}
& \textbf{Primal form ($\Omega=\nabla d(\mathcal X')$)} \\
\midrule
Euclidean
& $\tfrac12 x^2$
& $1$
& $1$ \\[0.3em]

KL
& $x\log x-x$
& $\displaystyle
\sup_{w_1,w_2\in\Omega}\max_i
\left(\frac{e^{(w_2)_i}}{e^{(w_1)_i}}\right)^{1/2}$
& $\displaystyle
\sup_{x_1,x_2\in\mathcal X'}\max_i
\left(\frac{(x_2)_i}{(x_1)_i}\right)^{1/2}$ \\[0.9em]

Burg
& $-\log x$
& $\displaystyle
\sup_{w_1,w_2\in\Omega}\max_i
\frac{(w_1)_i}{(w_2)_i}$
& $\displaystyle
\sup_{x_1,x_2\in\mathcal X'}\max_i
\frac{(x_2)_i}{(x_1)_i}$ \\[0.9em]

Power-type
& $\tfrac{|x|^p}{p}$
& $\displaystyle
\sup_{w_1,w_2\in\Omega}\max_i
\left(\frac{|(w_2)_i|}{|(w_1)_i|}\right)^{\frac{2-p}{2(p-1)}}$
& $\displaystyle
\sup_{x_1,x_2\in\mathcal X'}\max_i
\left(\frac{|(x_1)_i|}{|(x_2)_i|}\right)^{\frac{p-2}{2}}$ \\
\bottomrule
\end{tabular}
\caption{Dual norm distortion factor $\rho_{d^*}(\Omega)$ for standard separable geometries, in dual and primal coordinates, where \(\mathcal X'\subset\operatorname{int}(\dom d)\) is the primal region with \(\nabla d(\mathcal X')=\Omega\). The factor is finite when the coordinate ratios over $\Omega$ are bounded, but diverges as a coordinate approaches the boundary of the domain for the KL and Burg geometries, or the origin for the power-type geometry. In the Euclidean case, it equals $1$ on every $\Omega$.}
\label{tab:rho_examples}
\end{table}

\begin{restatable}[BLDF under smoothness and strong convexity]{lemma}{BLDFSmoothStronglyConvex}
\label{lem:bldf_smooth_strongly_convex}
Assume that \(d^*\) is \(L_{d^*}\)-smooth and \(\mu_{d^*}\)-strongly convex on a convex set
\(\Omega\subset \operatorname{int}(\dom d^*)\). Then, for every \(s>0\),
\[
c_{d^*}(s,\Omega;2)
\le
\sqrt{\frac{L_{d^*}}{\mu_{d^*}}}.
\]
In particular, if \(d^*(u)=\frac12\|u\|^2\), then
\[
c_{d^*}(s,\Omega;2)=1.
\]
\end{restatable}
\begin{proof}
By \(L_{d^*}\)-smoothness and \(\mu_{d^*}\)-strong convexity, for all admissible
\(u_1,u_2,v\),
\[
D_{d^*}(u_2+s v,u_2)
\le
\frac{L_{d^*}}{2}\|s v\|^2
=
s^2\frac{L_{d^*}}{2}\|v\|^2,
\]
and
\[
D_{d^*}(u_1+v,u_1)
\ge
\frac{\mu_{d^*}}{2}\|v\|^2.
\]
Therefore
\[
\frac{
D_{d^*}(u_2+s v,u_2)
}{
s^2D_{d^*}(u_1+v,u_1)
}
\le
\frac{L_{d^*}}{\mu_{d^*}}.
\]
Taking the square root and then the supremum over all admissible
\(u_1,u_2,v\) gives
\[
c_{d^*}(s,\Omega;2)
\le
\sqrt{\frac{L_{d^*}}{\mu_{d^*}}}.
\]
For \(d^*(u)=\frac12\|u\|^2\),
\[
D_{d^*}(u_2+s v,u_2)=\frac12s^2\|v\|^2,
\qquad
D_{d^*}(u_1+v,u_1)=\frac12\|v\|^2,
\]
so the ratio is exactly \(1\), hence
\[
c_{d^*}(s,\Omega;2)=1.
\]
\end{proof}

\begin{corollary}[Separable case]\label{cor:rho_separable}
Assume that $d^*$ is twice differentiable and separable on $\Omega$, i.e.,
\[
d^*(u)=\sum_{i=1}^n d_i^*(u_i).
\]
Then the dual norm distortion factor satisfies
\[
\rho_{d^*}(\Omega)
=
\sup_{w_1,w_2\in\Omega}
\max_{1\le i\le n}
\left(
\frac{(d_i^*)''((w_2)_i)}{(d_i^*)''((w_1)_i)}
\right)^{1/2}.
\]
\end{corollary}

\begin{proof}
For separable $d^*$, \cref{eq:primal_dual_norm} gives
\[
\|\xi\|_{w,*}^2=\tfrac12\sum_{i=1}^n (d_i^*)''(w_i)\,\xi_i^2 ,
\]
and the factor $\tfrac12$ cancels in the ratio below.
Hence, for any $w_1,w_2\in\Omega$,
\[
\sup_{\xi\neq 0}\frac{\|\xi\|_{w_2,*}}{\|\xi\|_{w_1,*}}
=
\max_{1\le i\le n}
\left(
\frac{(d_i^*)''((w_2)_i)}{(d_i^*)''((w_1)_i)}
\right)^{1/2},
\]
and the result follows by taking the supremum over $w_1,w_2\in\Omega$.
\end{proof}

\DualBLDFTwiceDiff*

\begin{proof}
Fix admissible $u_1,u_2\in\Omega$ and $v\neq 0$. Since $u_1+[0,1]v\subset\Omega$ and
$u_2+[0,s]v\subset\Omega$, the integral representation of the Bregman divergence gives
\[
D_{d^*}(u_1+v,u_1)
=
\int_0^1 (1-r)\, v^\top \nabla^2 d^*(u_1+r v)\, v\,dr,
\]
and this holds under pointwise twice differentiability alone. Indeed, \(g(r):=d^*(u_1+rv)\) is
convex, so \(g'\) is nondecreasing, and \(g'\) is differentiable everywhere on \([0,1]\) with
\(g''\ge0\) and \(\int_0^1g''\le g'(1)-g'(0)<\infty\); a function differentiable everywhere with
integrable derivative is absolutely continuous, so the fundamental theorem of calculus applies to
\(g'\), and one integration by parts produces the weight \((1-r)\). Note also that
\(D_{d^*}(u_1+v,u_1)>0\): \(d\) is Legendre, hence so is \(d^*\), hence \(d^*\) is strictly convex
on the convex set \(\Omega\subseteq\operatorname{int}(\dom d^*)\), so the denominators below do not
vanish. Explicitly,
\[
D_{d^*}(u_1+v,u_1)
=
\int_0^1 (1-r)\, v^\top \nabla^2 d^*(u_1+r v)\, v\,dr,
\]
and
\[
D_{d^*}(u_2+s v,u_2)
=
s^2\int_0^1 (1-r)\, v^\top \nabla^2 d^*(u_2+r s v)\, v\,dr.
\]
Both integrands are nonnegative and $\int_0^1(1-r)\,dr=\tfrac12$, so each divergence lies between
the extreme values of $\|v\|_{\cdot,*}^2$ on its segment,
\[
\inf_{w\in u_1+[0,1]v}\|v\|_{w,*}^2
\;\le\;
D_{d^*}(u_1+v,u_1)
\;\le\;
\sup_{w\in u_1+[0,1]v}\|v\|_{w,*}^2,
\]
and likewise \(D_{d^*}(u_2+s v,u_2)\in s^2[\inf,\sup]\) over $u_2+[0,s]v$. No continuity of
$\nabla^2d^*$ is required. Therefore,
\[
\left(
\frac{D_{d^*}(u_2+s v,u_2)}
{s^\gamma D_{d^*}(u_1+v,u_1)}
\right)^{1/2}
\le
s^{1-\gamma/2}
\left(
\frac{\sup_{w_2\in u_2+[0,s]v}\|v\|_{w_2,*}^2}
     {\inf_{w_1\in u_1+[0,1]v}\|v\|_{w_1,*}^2}
\right)^{1/2}
\le
s^{1-\gamma/2}\rho_{d^*}(\Omega),
\]
the last step because both segments lie in $\Omega$. Taking the supremum over all admissible
$u_1,u_2,v$ yields the claim.
\end{proof}

With the following additional regularity assumptions, the two factors are equal.

\begin{restatable}{proposition}{DualBLDFEquality}\label{prop:dual_bldf_equality}
Assume that $\Omega$ is open and that $\nabla^2d^*$ is continuous and positive definite
on $\Omega$. Then
\[
c_{d^*}(s,\Omega;2)=\rho_{d^*}(\Omega)
\quad\text{for every }s>0,
\qquad
c_{d^*}(s,\Omega;\gamma)=s^{1-\gamma/2}\rho_{d^*}(\Omega)
\quad\text{for every }\gamma>0 .
\]
\end{restatable}

\begin{proof}
The proof of \cref{prop:dual_bldf_twice_diff} gives \(\le\) at every \(\gamma\), although the
proposition is stated at \(\gamma=2\); only the reverse direction is at issue. Fix
\(u_1,u_2\in\Omega\) and \(\xi\neq0\), and take \(v=\varepsilon\xi\). Since \(\Omega\) is open,
both \(u_1+[0,1]v\) and \(u_2+[0,s]v\) lie in \(\Omega\) once \(\varepsilon\) is small enough, so
the pair is admissible. Here \(\nabla^2d^*\) is continuous, so the weighted mean-value theorem
applies to the integral representation used in the proof of \cref{prop:dual_bldf_twice_diff} and
gives points \(\bar w_1(\varepsilon)\in u_1+[0,\varepsilon]\xi\) and \(\bar w_2(\varepsilon)\in
u_2+[0,s\varepsilon]\xi\) with \(D_{d^*}(u_1+v,u_1)=\|v\|_{\bar w_1(\varepsilon),*}^2\) and
\(D_{d^*}(u_2+sv,u_2)=s^2\|v\|_{\bar w_2(\varepsilon),*}^2\). Both points converge to \(u_1\) and
\(u_2\) as \(\varepsilon\downarrow0\), and \(\nabla^2d^*\) is positive definite there, so
\[
\left(
\frac{D_{d^*}(u_2+s v,u_2)}
{s^\gamma D_{d^*}(u_1+v,u_1)}
\right)^{1/2}
=
s^{1-\gamma/2}\frac{\|\xi\|_{\bar w_2(\varepsilon),*}}{\|\xi\|_{\bar w_1(\varepsilon),*}}
\xrightarrow[\varepsilon\downarrow0]{}
s^{1-\gamma/2}\frac{\|\xi\|_{u_2,*}}{\|\xi\|_{u_1,*}} .
\]
The left-hand side is bounded above by \(c_{d^*}(s,\Omega;\gamma)\) for every \(\varepsilon\), so
the limit is too. Taking the supremum over \(u_1,u_2\in\Omega\) and \(\xi\neq0\) gives
\(c_{d^*}(s,\Omega;\gamma)\ge s^{1-\gamma/2}\rho_{d^*}(\Omega)\).
\end{proof}

\clearpage

\subsection{Technical results on Accelerated Method}

Throughout this subsection, \(\alpha_k:=\mu+\eta_k^{-1}\).

\begin{restatable}[Implicit recursive accelerated bound]{lemma}{BoundAccImplicit}
\label{lem:bound_acc_implicit}
Assume that \(t_k=1\), or that
\[
\phi_k^{\mathrm{best}}\le \mathcal D^{\eta_k}(\lambda_k).
\]
If the inner loop of \Cref{algo:bregman_adaptive} exits, i.e.,
\[
\phi_{k+1}^{\mathrm{best}}
\le
(1-t_k)\phi_k^{\mathrm{best}}
+
t_k\phi_k^{\mathrm{low}}
-
\left(\mu+\eta_{k+1}^{-1}\right)D_d(z_k,z_{k+1}),
\]
then
\[
\phi_{k+1}^{\mathrm{best}}
=
\phi(x_{k+1}^{\mathrm{best}})
\le
\phi_\star^{\eta_{k+1}}.
\]
\end{restatable}

\begin{proof}
Apply \cref{lem:pd_lower_step} with \(b_k=\phi_k^{\mathrm{best}}\); when \(t_k=1\), apply instead
its second conclusion, which carries no hypothesis on \(\phi_k^{\mathrm{best}}\) and whose
right-hand side is the same because \((1-t_k)=0\). The exit condition gives
\[
\phi_{k+1}^{\mathrm{best}}
\le
\mathcal D^{\eta_{k+1}}(\lambda_{k+1}).
\]
Since \(S^{\eta_{k+1}}(x,\lambda_{k+1})\le \phi^{\eta_{k+1}}(x)\) for all \(x\),
\[
\mathcal D^{\eta_{k+1}}(\lambda_{k+1})
=
\min_x S^{\eta_{k+1}}(x,\lambda_{k+1})
\le
\min_x \phi^{\eta_{k+1}}(x)
=
\phi_\star^{\eta_{k+1}}.
\]
Therefore
\[
\phi_{k+1}^{\mathrm{best}}
\le
\phi_\star^{\eta_{k+1}}.
\]
\end{proof}

\GenericAcceleratedRateMk*

\begin{proof}
By \cref{lem:bound_acc_implicit}, the invariant and the exit condition imply
\[
\phi_{k+1}^{\mathrm{best}}
\le
\phi_\star^{\eta_{k+1}}.
\]
Thus the iterates satisfy the implicit descent condition. Applying
\cref{thm:implicit_descent_rate} with \(x_{k+1}=x_{k+1}^{\mathrm{best}}\) gives
\[
\phi_{k+1}^{\mathrm{best}}-\phi_\star
\le
\eta_{k+1}^{-1}D_d(x_\star,x_0).
\]
The monotonicity
\[
\phi_{k+1}^{\mathrm{best}}\le \phi_k^{\mathrm{best}}
\]
follows from the definition of the best value.

It remains to estimate \(\eta_k^{-1}\). First assume \(\mu=0\). Then
\[
\eta_{k+1}^{-1}=M_kt_k^2,
\qquad
t_k=1-\frac{\eta_k}{\eta_{k+1}}.
\]
Hence
\[
\eta_{k+1}
=
M_k(\eta_{k+1}-\eta_k)^2,
\]
so
\[
\eta_{k+1}-\eta_k
=
\sqrt{\frac{\eta_{k+1}}{M_k}}.
\]
Therefore
\[
\sqrt{\eta_{k+1}}-\sqrt{\eta_k}
=
\frac{\eta_{k+1}-\eta_k}
{\sqrt{\eta_{k+1}}+\sqrt{\eta_k}}
\ge
\frac{1}{2\sqrt{M_k}}.
\]
With the limiting initialization \(\eta_0=0\), summing gives
\[
\sqrt{\eta_k}
\ge
\frac12\sum_{i=0}^{k-1}M_i^{-1/2},
\]
hence
\[
\eta_k^{-1}
\le
\frac{4}{
\left(\sum_{i=0}^{k-1}M_i^{-1/2}\right)^2}.
\]
If \(M_i\le M\) for all \(i\ge K\), then
\[
\sqrt{\eta_k}
\ge
\sqrt{\eta_K}+\frac{k-K}{2\sqrt M},
\]
and thus
\[
\eta_k^{-1}
\le
\frac{\eta_K^{-1}}{
\left(
1+\frac{k-K}{2\sqrt{M\eta_K}}
\right)^2}.
\]

Now assume \(\mu>0\). Since
\[
\mu+\eta_{k+1}^{-1}=M_kt_k^2,
\]
we have
\[
t_k\ge \sqrt{\frac{\mu}{M_k}}.
\]
Thus
\[
\eta_{k+1}^{-1}
=
(1-t_k)\eta_k^{-1}
\le
\left(1-\sqrt{\frac{\mu}{M_k}}\right)\eta_k^{-1}.
\]
With the limiting initialization \(\eta_0=0\), we have \(t_0=1\), hence
\[
\eta_1^{-1}=M_0-\mu.
\]
Iterating from \(i=1\) to \(k-1\) gives
\[
\eta_k^{-1}
\le
(M_0-\mu)
\prod_{i=1}^{k-1}
\left(1-\sqrt{\frac{\mu}{M_i}}\right).
\]
If \(M_i\le M\) for all \(i\ge K\), then
\[
1-\sqrt{\frac{\mu}{M_i}}
\le
1-\sqrt{\frac{\mu}{M}},
\]
and therefore
\[
\eta_k^{-1}
\le
\eta_K^{-1}
\left(
1-\sqrt{\frac{\mu}{M}}
\right)^{k-K}.
\]
\end{proof}

\begin{restatable}[Generic accelerated bound]{lemma}{BoundAccGeneric}
\label{lem:bound_acc_generic}
Assume that \(t_k=1\), or that
\[
\phi(x_k)\le \mathcal D^{\eta_k}(\lambda_k),
\qquad
\lambda_{k+1}=(1-t_k)\lambda_k+t_k\nabla f(y_k),
\qquad
\eta_{k+1}^{-1}=(1-t_k)\eta_k^{-1},
\]
and set
\[
y_k=(1-\tau_k)x_k+\tau_kz_k.
\]
Let
\[
\mathcal L_{k+1}(z_{k+1})
:=
(1-t_k)\phi(x_k)+t_k\phi_k^{\mathrm{low}}-(\mu+\eta_{k+1}^{-1})D_d(z_k,z_{k+1}),
\]
with \(\phi_k^{\mathrm{low}}\) as in \cref{lem:pd_lower_step}. This is the right-hand side of the
exit condition of \cref{algo:bregman_adaptive} with \(\phi_k^{\mathrm{best}}=\phi(x_k)\), and
\(\mathcal L_{k+1}(z_{k+1})\le\mathcal D^{\eta_{k+1}}(\lambda_{k+1})\) by \cref{lem:pd_lower_step}.
Then \(x_{k+1}\) in \(\Cref{algo:accstep}\) satisfies
\[
\begin{aligned}
\mathcal L_{k+1}(z_{k+1})
\ge\;&
\phi(x_{k+1})
+
LD_d(y_k,x_{k+1})
-
(\mu+\eta_{k+1}^{-1})D_d(z_k,z_{k+1})
\\
&+
(1-t_k)\mu D_d(x_k,y_k)
+
t_k\mu D_d(z_k,y_k)
\\
&+
(\tau_k-t_k)
\langle\nabla\phi(y_k),x_k-z_k\rangle .
\end{aligned}
\]
\end{restatable}

\begin{proof}
The bound \(\mathcal L_{k+1}(z_{k+1})\le\mathcal D^{\eta_{k+1}}(\lambda_{k+1})\) is
\cref{lem:pd_lower_step} with \(b_k=\phi(x_k)\), or its second conclusion when \(t_k=1\). When
\(t_k=1\), every term below carrying the factor \((1-t_k)\) — including the strong convexity
substitution for \(\phi(x_k)\) — is absent. Expanding \(\phi_k^{\mathrm{low}}\) gives
\[
\begin{aligned}
\mathcal L_{k+1}(z_{k+1})
=\;&
(1-t_k)\phi(x_k)
+
t_k\phi(y_k)
+
t_k\langle\nabla\phi(y_k),z_k-y_k\rangle
\\
&+
t_k\mu D_d(z_k,y_k)
-
(\mu+\eta_{k+1}^{-1})D_d(z_k,z_{k+1}).
\end{aligned}
\]
By relative \(\mu\)-strong convexity of \(\phi\),
\[
\phi(x_k)
\ge
\phi(y_k)
+
\langle\nabla\phi(y_k),x_k-y_k\rangle
+
\mu D_d(x_k,y_k).
\]
Substituting this lower bound yields
\[
\begin{aligned}
\mathcal L_{k+1}(z_{k+1})
\ge\;&
\phi(y_k)
+
(1-t_k)\mu D_d(x_k,y_k)
+
t_k\mu D_d(z_k,y_k)
\\
&+
\left\langle
\nabla\phi(y_k),
(1-t_k)(x_k-y_k)+t_k(z_k-y_k)
\right\rangle
\\
&-
(\mu+\eta_{k+1}^{-1})D_d(z_k,z_{k+1}).
\end{aligned}
\]
Since
\[
y_k=(1-\tau_k)x_k+\tau_kz_k,
\]
we have
\[
(1-t_k)(x_k-y_k)+t_k(z_k-y_k)
=
(\tau_k-t_k)(x_k-z_k).
\]
Finally, the descent step gives
\[
\phi(x_{k+1})
\le
\phi(y_k)-LD_d(y_k,x_{k+1}),
\]
which gives the claim.
\end{proof}

\begin{restatable}[Convex specialization]{lemma}{BoundAccConvex}
\label{lem:bound_acc_convex}
Assume \(\mu=0\) and choose \(\tau_k=t_k\). Then
\[
\mathcal L_{k+1}(z_{k+1})
\ge
\phi(x_{k+1})
+
LD_d(y_k,x_{k+1})
-
\alpha_{k+1}D_d(z_k,z_{k+1}).
\]
\end{restatable}

\begin{proof}
Apply \cref{lem:bound_acc_generic}. Since \(\mu=0\), the two terms multiplied by
\(\mu\) vanish. Since \(\tau_k=t_k\), the scalar-product term also vanishes.
\end{proof}

\begin{restatable}[Convex one-step bound under dual norm distortion]{proposition}{ConvexOneStepDualDistortion}
\label{prop:convex_rate_dual_distortion}
Under \cref{ass:setting}, assume \(\mu=0\). Let \(\rho_k\ge1\) be such that the relevant dual segments of
iteration \(k\) are contained in a convex set \(\Omega_k\) satisfying
\[
    \rho_{d^*}(\Omega_k)\le \rho_k .
\]
Let \(M_k\ge\rho_k^2L\) and choose
\[
    \tau_k=t_k,
    \qquad
    \eta_{k+1}^{-1}=M_k t_k^2,
    \qquad
    \eta_{k+1}^{-1}=(1-t_k)\eta_k^{-1}.
\]
Then the accelerated step satisfies
\[
    \mathcal L_{k+1}(z_{k+1})
    \ge
    \phi(x_{k+1}).
\]
Hence the exit condition holds at every acceleration parameter
\(M_k\ge\rho_k^2L\), with \(\mathcal L_{k+1}(z_{k+1})\) as in \cref{lem:bound_acc_generic}.
\end{restatable}

\begin{proof}
By \cref{lem:bound_acc_convex},
\[
\mathcal L_{k+1}(z_{k+1})
\ge
\phi(x_{k+1})
+
LD_d(y_k,x_{k+1})
-
\eta_{k+1}^{-1}D_d(z_k,z_{k+1}).
\]
Since \(\mu=0\),
\[
\nabla d(x_{k+1})
=
\nabla d(y_k)-\frac1L\nabla\phi(y_k),
\qquad
\nabla d(z_{k+1})
=
\nabla d(z_k)-\frac{t_k}{\eta_{k+1}^{-1}}\nabla\phi(y_k).
\]
Thus the two dual displacements are parallel, with relative scaling
\[
    \frac{Lt_k}{\eta_{k+1}^{-1}}.
\]
Using \(\cref{prop:dual_bldf_twice_diff}\) and
\(\rho_{d^*}(\Omega_k)\le\rho_k\),
\[
D_d(z_k,z_{k+1})
\le
\rho_k^2
\left(\frac{Lt_k}{\eta_{k+1}^{-1}}\right)^2
D_d(y_k,x_{k+1}).
\]
Therefore
\[
\begin{aligned}
LD_d(y_k,x_{k+1})
-
\eta_{k+1}^{-1}D_d(z_k,z_{k+1})
\ge
\left(
L-\frac{\rho_k^2L^2t_k^2}{\eta_{k+1}^{-1}}
\right)
D_d(y_k,x_{k+1})
\ge
0,
\end{aligned}
\]
where the last inequality uses
\(\eta_{k+1}^{-1}=M_kt_k^2\ge\rho_k^2Lt_k^2\). Hence
\[
\mathcal L_{k+1}(z_{k+1})
\ge
\phi(x_{k+1}).
\]
\end{proof}

\begin{restatable}[Strongly convex specialization]{lemma}{BoundAccStronglyConvex}
\label{lem:bound_acc_strongly_convex}
Assume \(\mu>0\). Define
\[
q_{k+1}
:=
\nabla d^*\!\left(
\nabla d(y_k)-\frac1\mu\nabla\phi(y_k)
\right).
\]
Assume that \(\nabla d(y_k)-\frac1\mu\nabla\phi(y_k)\in\operatorname{int}(\dom d^*)\), so that
\(q_{k+1}\) is well defined, and that \(\tau_k<1\). Then, if
\[
t_k\le 1,
\qquad
\frac{t_k-\tau_k}{1-\tau_k}\le t_k,
\]
we have
\[
\begin{aligned}
\mathcal L_{k+1}(z_{k+1})
\ge\;&
\phi(x_{k+1})
+
\Bigg[
LD_d(y_k,x_{k+1})
-
\mu\frac{t_k-\tau_k}{1-\tau_k}
D_d(y_k,q_{k+1})
\Bigg]
\\
&+
\Bigg[
\mu\frac{t_k-\tau_k}{1-\tau_k}
D_d(z_k,q_{k+1})
-
\alpha_{k+1}D_d(z_k,z_{k+1})
\Bigg].
\end{aligned}
\]
\end{restatable}

\begin{proof}
By definition of \(q_{k+1}\),
\[
\nabla\phi(y_k)
=
\mu\big(\nabla d(y_k)-\nabla d(q_{k+1})\big).
\]
Since
\[
y_k=(1-\tau_k)x_k+\tau_kz_k,
\]
we have
\[
x_k-z_k=\frac{y_k-z_k}{1-\tau_k}.
\]
Therefore
\[
\begin{aligned}
(\tau_k-t_k)
\langle\nabla\phi(y_k),x_k-z_k\rangle
&=
\frac{\mu(\tau_k-t_k)}{1-\tau_k}
\left\langle
\nabla d(y_k)-\nabla d(q_{k+1}),
y_k-z_k
\right\rangle .
\end{aligned}
\]
By the three-point identity,
\[
\left\langle
\nabla d(y_k)-\nabla d(q_{k+1}),
y_k-z_k
\right\rangle
=
D_d(y_k,q_{k+1})
+
D_d(z_k,y_k)
-
D_d(z_k,q_{k+1}).
\]
Hence
\[
\begin{aligned}
(\tau_k-t_k)
\langle\nabla\phi(y_k),x_k-z_k\rangle
=
&-\mu\frac{t_k-\tau_k}{1-\tau_k}D_d(y_k,q_{k+1})
\\
&-\mu\frac{t_k-\tau_k}{1-\tau_k}D_d(z_k,y_k)
\\
&+\mu\frac{t_k-\tau_k}{1-\tau_k}D_d(z_k,q_{k+1}).
\end{aligned}
\]
Substituting this identity into \cref{lem:bound_acc_generic} and writing \(\alpha_{k+1}=\mu+\eta_{k+1}^{-1}\) gives
\[
\begin{aligned}
\mathcal L_{k+1}(z_{k+1})
\ge\;&
\phi(x_{k+1})
+
LD_d(y_k,x_{k+1})
-
\alpha_{k+1}D_d(z_k,z_{k+1})
\\
&+
(1-t_k)\mu D_d(x_k,y_k)
+
t_k\mu D_d(z_k,y_k)
\\
&-
\mu\frac{t_k-\tau_k}{1-\tau_k}D_d(y_k,q_{k+1})
-\mu\frac{t_k-\tau_k}{1-\tau_k}D_d(z_k,y_k)
\\
&+
\mu\frac{t_k-\tau_k}{1-\tau_k}D_d(z_k,q_{k+1}).
\end{aligned}
\]
Rearranging,
\[
\begin{aligned}
\mathcal L_{k+1}(z_{k+1})
\ge\;&
\phi(x_{k+1})
+
\Bigg[
LD_d(y_k,x_{k+1})
-
\mu\frac{t_k-\tau_k}{1-\tau_k}
D_d(y_k,q_{k+1})
\Bigg]
\\
&+
\Bigg[
\mu\frac{t_k-\tau_k}{1-\tau_k}
D_d(z_k,q_{k+1})
-
\alpha_{k+1}D_d(z_k,z_{k+1})
\Bigg]
\\
&+
(1-t_k)\mu D_d(x_k,y_k)
+
\mu\left(
t_k-\frac{t_k-\tau_k}{1-\tau_k}
\right)D_d(z_k,y_k).
\end{aligned}
\]
The last two terms are nonnegative under
\[
t_k\le 1,
\qquad
\frac{t_k-\tau_k}{1-\tau_k}\le t_k.
\]
Dropping them gives the result.
\end{proof}

\begin{restatable}[Strongly convex one-step bound under dual norm distortion]{proposition}{StronglyConvexOneStepDualDistortion}
\label{prop:strongly_convex_rate_dual_distortion}
Under \cref{ass:setting}, assume \(\mu>0\) and that \(q_{k+1}\) of
\cref{lem:bound_acc_strongly_convex} is well defined. Let \(\rho_k\ge1\) be such that the relevant dual segments of
iteration \(k\), including the dual segments joining \(\nabla d(y_k)\) to \(\nabla d(q_{k+1})\) and
joining \(\nabla d(z_k)\) to \(\nabla d(q_{k+1})\), are contained in a convex set \(\Omega_k\) satisfying
\[
    \rho_{d^*}(\Omega_k)\le \rho_k .
\]
Let \(M_k\ge\rho_k^4L\), and let \(t_k\in(0,1]\) be chosen by
\[
    \mu+\eta_{k+1}^{-1}
    =
    M_kt_k^2,
    \qquad
    \eta_{k+1}^{-1}
    =
    (1-t_k)\eta_k^{-1}.
\]
Equivalently, with \(\alpha_k:=\mu+\eta_k^{-1}\), \(t_k\) is the positive root of
\[
    M_kt_k^2+(\alpha_k-\mu)t_k-\alpha_k=0.
\]
Set
\[
    \tau_k
    =
    \frac{t_k-\mu/\sqrt{M_kL}}
    {1-\mu/\sqrt{M_kL}},
    \qquad
    y_k=(1-\tau_k)x_k+\tau_kz_k .
\]
Throughout, \(L\) is the constant the gradient step of \cref{algo:accstep} actually uses, which is
the same constant that forms \(\tau_k\).
Then the accelerated step satisfies
\[
    \mathcal L_{k+1}(z_{k+1})
    \ge
    \phi(x_{k+1}).
\]
Hence the exit condition holds at every acceleration parameter
\(M_k\ge\rho_k^4L\), with \(\mathcal L_{k+1}(z_{k+1})\) as in \cref{lem:bound_acc_generic}.
\end{restatable}

\begin{proof}
Assume first \(t_k<1\); the case \(t_k=1\), which the initialization imposes at \(k=0\), is treated
at the end.

Set
\[
    \alpha_{k+1}:=\mu+\eta_{k+1}^{-1},
    \qquad
    \delta_k:=\frac{t_k-\tau_k}{1-\tau_k}.
\]
By construction,
\[
    \delta_k=\frac{\mu}{\sqrt{M_kL}}.
\]
Moreover, \(\alpha_{k+1}=M_kt_k^2\).

By \cref{lem:bound_acc_strongly_convex},
\[
\begin{aligned}
\mathcal L_{k+1}(z_{k+1})
\ge\;&
\phi(x_{k+1})
+
\Big[
LD_d(y_k,x_{k+1})
-
\mu\delta_kD_d(y_k,q_{k+1})
\Big]
\\
&+
\Big[
\mu\delta_kD_d(z_k,q_{k+1})
-
\alpha_{k+1}D_d(z_k,z_{k+1})
\Big],
\end{aligned}
\]
provided \(t_k\le1\) and \(\delta_k\le t_k\).

We first show that the two brackets are nonnegative. The two uses of the dual
BLDF are separate and each compares displacements along the same dual ray.
For the first bracket, set
\[
    u_y:=\nabla d(y_k),
    \qquad
    v_y:=\nabla d(x_{k+1})-\nabla d(y_k).
\]
The primal gradient step and the definition of \(q_{k+1}\) give
\[
    \nabla d(x_{k+1})=u_y+v_y,
    \qquad
    \nabla d(q_{k+1})=u_y+\frac{L}{\mu}v_y .
\]
Hence
\[
D_d(y_k,x_{k+1})=D_{d^*}(u_y+v_y,u_y),
\qquad
D_d(y_k,q_{k+1})=D_{d^*}\left(u_y+\frac{L}{\mu}v_y,u_y\right).
\]
Applying the dual BLDF on this same ray, with scaling \(L/\mu\), gives
\[
D_d(y_k,x_{k+1})
\ge
\frac1{\rho_k^2}
\left(\frac{\mu}{L}\right)^2
D_d(y_k,q_{k+1}),
\]
hence
\[
LD_d(y_k,x_{k+1})
-
\mu\delta_kD_d(y_k,q_{k+1})
\ge
\left(
\frac{\mu^2}{\rho_k^2L}
-
\mu\delta_k
\right)
D_d(y_k,q_{k+1})
\ge
0,
\]
since \(\sqrt{M_kL}\ge\rho_k^2L\).
For the second bracket, set
\[
    u_z:=\nabla d(z_k),
    \qquad
    v_z:=\nabla d(q_{k+1})-\nabla d(z_k).
\]
Using the optimality conditions of the dual map and
\(\alpha_{k+1}=\mu+\eta_{k+1}^{-1}\), we get
\[
    \nabla d(z_{k+1})
    =
    u_z+\frac{t_k\mu}{\alpha_{k+1}}v_z .
\]
Thus
\[
D_d(z_k,z_{k+1})
= D_{d^*}\left(u_z+\frac{t_k\mu}{\alpha_{k+1}}v_z,u_z\right),
\qquad
D_d(z_k,q_{k+1})=D_{d^*}(u_z+v_z,u_z),
\]
and the dual BLDF on this ray gives
\[
D_d(z_k,z_{k+1})
\le
\rho_k^2
\left(\frac{t_k\mu}{\alpha_{k+1}}\right)^2
D_d(z_k,q_{k+1}),
\]
so
\[
\begin{aligned}
\mu\delta_kD_d(z_k,q_{k+1})
-
\alpha_{k+1}D_d(z_k,z_{k+1})
\ge
\left(
\mu\delta_k
-
\frac{\rho_k^2t_k^2\mu^2}{\alpha_{k+1}}
\right)
D_d(z_k,q_{k+1})
\ge
0,
\end{aligned}
\]
where the last inequality uses \(\delta_k=\mu/\sqrt{M_kL}\) and \(\alpha_{k+1}=M_kt_k^2\), so that
\[
    \mu\delta_k-\frac{\rho_k^2t_k^2\mu^2}{\alpha_{k+1}}
    =
    \mu^2\left(\frac{1}{\sqrt{M_kL}}-\frac{\rho_k^2}{M_k}\right)
    \ge 0
    \quad\text{since } M_k\ge\rho_k^4L .
\]
Thus both brackets are nonnegative, and therefore
\[
\mathcal L_{k+1}(z_{k+1})
\ge
\phi(x_{k+1}).
\]

It remains to check the side conditions of
\(\cref{lem:bound_acc_strongly_convex}\). Since
\[
    \alpha_{k+1}
    =
    M_kt_k^2
    =
    \mu+\eta_{k+1}^{-1}
    \ge \mu,
\]
we have
\[
    t_k
    \ge
    \sqrt{\frac{\mu}{M_k}}
    \ge
    \sqrt{\frac{\mu}{M_k}}\sqrt{\frac{\mu}{L}}
    =
    \frac{\mu}{\sqrt{M_kL}}
    =
    \delta_k,
\]
where we used \(\mu\le L\). Also \(t_k\le1\): the polynomial
\[
    p(t)=M_kt^2+(\alpha_k-\mu)t-\alpha_k
\]
satisfies \(p(0)=-\alpha_k<0\) and
\[
    p(1)=M_k-\mu\ge0,
\]
so its positive root lies in \((0,1]\). Hence the assumptions of
\(\cref{lem:bound_acc_strongly_convex}\) hold.

\paragraph{The case \(t_k=1\).}
This case is outside \cref{lem:bound_acc_strongly_convex}, which requires \(\tau_k<1\): the
definition of \(\tau_k\) gives \(\tau_k=1\) when \(t_k=1\), and \(\delta_k=(t_k-\tau_k)/(1-\tau_k)\)
is then \(0/0\). It arises only at \(k=0\), where \cref{algo:bregman_adaptive} imposes \(t_0=1\) from
outside the root formula above, and it does not need the virtual point \(q_{k+1}\) at all. The
initialization sets \(x_0^{\mathrm{best}}=z_0=x_0\), so \(y_0=x_0\) for every \(\tau_0\), and
\(\nabla f(y_0)=\nabla\phi(x_0)\). The two dual displacements therefore leave the single anchor
\(\nabla d(x_0)\) along \(-\nabla\phi(x_0)\):
\[
    \nabla d(x_1)=\nabla d(x_0)-\tfrac1L\nabla\phi(x_0),
    \qquad
    \nabla d(z_1)=\nabla d(x_0)-\tfrac1{\alpha_1}\nabla\phi(x_0),
\]
which is the configuration of \cref{prop:convex_rate_dual_distortion}. At \(y_0=z_0=x_0\), the
\(t_k=1\) case of \cref{lem:pd_lower_step} gives
\(\mathcal D^{\eta_1}(\lambda_1)\ge\min_x\{\phi(x_0)+\langle\nabla\phi(x_0),x-x_0\rangle
+\alpha_1D_d(x,x_0)\}\). The minimizer is \(z_1\), and
\(\nabla\phi(x_0)=\alpha_1(\nabla d(x_0)-\nabla d(z_1))\) gives
\(\langle\nabla\phi(x_0),z_1-x_0\rangle=-\alpha_1\big(D_d(z_1,x_0)+D_d(x_0,z_1)\big)\). Hence
\[
    \mathcal D^{\eta_1}(\lambda_1)
    \ge
    \phi(x_0)-\alpha_1D_d(x_0,z_1)
    =
    \phi(x_0)-\alpha_1D_d(z_0,z_1)
    =
    \mathcal L_1(z_1),
\]
since \(\phi_0^{\mathrm{low}}=\phi(x_0)\) at \(y_0=z_0=x_0\).
The same computation with \(L\) in place of \(\alpha_1\) and \(x_1\) in place of \(z_1\), together
with the relative smoothness of \(\phi\), gives \(\phi(x_1)\le\phi(x_0)-LD_d(x_0,x_1)\). One comparison of the two displacements at the scale \(s=L/\alpha_1\), through
\cref{def:dual_bldf} with \(\rho_{d^*}(\Omega_0)\le\rho_0\), gives
\(\alpha_1D_d(z_0,z_1)\le\rho_0^2L^2\alpha_1^{-1}D_d(y_0,x_1)\). The exit condition
\(\mathcal L_1(z_1)\ge\phi(x_1)\) follows whenever \(\alpha_1\ge\rho_0^2L\), and the
choice in the statement gives \(\alpha_1=M_0t_0^2=M_0\ge\rho_0^4L\ge\rho_0^2L\) since
\(\rho_0\ge1\). The first iteration therefore passes the exit condition at every acceleration
parameter \(M_0\ge\rho_0^4L\).
\end{proof}

\EventualBLDFRate*

\begin{proof}
We prove the convex case; the strongly convex case is identical, replacing
\(\rho^2L\) by \(\rho^4L\).

For every \(k\ge K\), \cref{prop:convex_rate_dual_distortion} with \(\rho_k\le\rho\) and
\(L_{\mathrm{cur}}\le2L\) implies that any trial value \(M\ge 2\rho^2L\) is sufficient for
acceptance. Since each trial doubles \(M_k\), the inner loop exits at every \(k\ge K\); with the
hypothesis for \(k<K\), it exits at every iteration, so \cref{thm:generic_accelerated_rate_Mk}
applies. By the update rule in
\cref{algo:bregman_adaptive}, the search starts from a reduced value of the
previously accepted \(M_{k-1}\) and then doubles until acceptance. Hence, if the
search starts below \(2\rho^2L\), the accepted value is at most \(4\rho^2L\). If
it starts above \(2\rho^2L\), the first trial is accepted, and the next iteration
again starts from a fixed reduction of that accepted value. Thus any overestimate
inherited from earlier iterations decreases geometrically until it falls below
\(4\rho^2L\). Therefore there exists \(\bar K\ge K\) such that
\[
    M_k\le 4\rho^2L
    \qquad \forall k\ge \bar K .
\]

Applying \cref{thm:generic_accelerated_rate_Mk} from iteration \(\bar K\) with
the uniform bound \(M_k\le 4\rho^2L\) gives
\[
\eta_k^{-1}
\le
\frac{\eta_{\bar K}^{-1}}{
\left(
1+
\frac{k-\bar K}{2\sqrt{4\rho^2L\eta_{\bar K}}}
\right)^2}
=
\frac{\eta_{\bar K}^{-1}}{
\left(
1+
\frac{k-\bar K}{4\rho\sqrt{L\eta_{\bar K}}}
\right)^2}.
\]
The strongly convex case follows similarly from
\(M_k\le 4\rho^4L\), yielding
\[
\eta_k^{-1}
\le
\eta_{\bar K}^{-1}
\left(
1-\sqrt{\frac{\mu}{4\rho^4L}}
\right)^{k-\bar K}.
\]
\end{proof}

\IterationComplexity*

\begin{proof}
Write \(D:=D_d(x_\star,x_0)\) and let \(k\ge\bar K\).

\emph{Case \(\mu=0\).} By \cref{thm:eventual_bldf_rate} it suffices that
\[
\frac{\eta_{\bar K}^{-1}D}{\left(1+s\right)^2}\le\varepsilon,
\qquad
s:=\frac{k-\bar K}{4\rho\sqrt{L\eta_{\bar K}}},
\]
that is \(1+s\ge\sqrt{\eta_{\bar K}^{-1}D/\varepsilon}\). Dropping the \(1\) on the left and
substituting \(s\) gives the sufficient condition
\[
k-\bar K
\;\ge\;
4\rho\sqrt{L\eta_{\bar K}}\,\sqrt{\frac{\eta_{\bar K}^{-1}D}{\varepsilon}}
\;=\;
4\rho\sqrt{\frac{LD}{\varepsilon}},
\]
in which \(\eta_{\bar K}\) cancels.

\emph{Case \(\mu>0\).} Set \(q:=\sqrt{\mu/(4\rho^4L)}\in(0,1)\). By
\cref{thm:eventual_bldf_rate} it suffices that
\(\eta_{\bar K}^{-1}D(1-q)^{k-\bar K}\le\varepsilon\), i.e.
\[
(k-\bar K)\log\frac{1}{1-q}\;\ge\;\log\frac{\eta_{\bar K}^{-1}D}{\varepsilon}.
\]
Since \(\log\frac{1}{1-q}\ge q\), the condition
\(k-\bar K\ge q^{-1}\log\!\left(\eta_{\bar K}^{-1}D/\varepsilon\right)\) is sufficient, and
\(q^{-1}=2\rho^2\sqrt{L/\mu}\).
\end{proof}

\subsection{Oracle complexity} \label{sec:oracle_complexity}
Backtracking uses a variable number of oracle calls per accepted iteration. Let \(L\) denote the
relative smoothness constant, let \(L_0\) and \(M_{-1}\) be the initial estimates, and let
\[
    M_{\max}:=\max_{0\le i<K} M_i
\]
be the largest accepted value in the first \(K\) outer iterations. Each iteration reduces the
current estimates before the searches double them. This gives constant amortized overhead, with
logarithmic terms for underestimated initial values. Writing \((a)_+:=\max\{a,0\}\), the total
number of trials satisfies
\[
    N_{\mathrm{trials}}
    \le
    4K
    +
    \left\lceil
    \log_2\!\left(\frac{M_{\max}}{M_{-1}}\right)_+
    \right\rceil
    +
    \left\lceil
    \log_2\!\left(\frac{L}{L_0}\right)_+
    \right\rceil .
\]
This is the standard backtracking tradeoff \citep{nesterov2006cubic}: the additional trials allow
the method to use local estimates instead of a fixed worst-case constant.

\clearpage
\section{Euclidean specialization and comparison with Nesterov}\label{app:euclidean_nesterov}

We show that, in the Euclidean geometry, ABrA-GD reduces to the standard
estimate-sequence form of Nesterov acceleration. Let
\[
    d(x)=\frac12\|x\|^2,
    \qquad
    D_d(x,y)=\frac12\|x-y\|^2,
    \qquad
    \nabla d(x)=x.
\]
Assume for this comparison that the line searches accept the sharp constants,
namely the smoothness estimate is \(L\) and the accepted acceleration parameter
is \(M_k=L\). Then the BLDF is equal to one, so no geometric distortion is
introduced by the Bregman geometry.

One difference from the ordinary trajectory remains. \Cref{algo:bregman_adaptive} passes
\(x_k^{\mathrm{best}}\), not \(x_k\), into the next step, so the identities below are the classical
estimate-sequence identities with \(x_k\) read as \(x_k^{\mathrm{best}}\). The correspondence is
therefore between the updates. The two trajectories coincide whenever the gradient step never
increases the objective, and ABrA-GD is otherwise a monotone variant of the classical method.

\paragraph{Convex case.}
When \(\mu=0\), one has \(\tau_k=t_k\) and
\(\eta_{k+1}^{-1}=Lt_k^2\). Writing
\(g_k=\nabla\phi(y_k)\), the ABrA-GD updates become
\[
    y_k=(1-t_k)x_k+t_k z_k,
    \qquad
    x_{k+1}=y_k-\frac1L g_k,
    \qquad
    z_{k+1}=z_k-\frac{1}{Lt_k}g_k .
\]
Therefore
\[
    x_{k+1}
    =
    (1-t_k)x_k+t_k z_{k+1},
\]
which is exactly the usual accelerated estimate-sequence coupling. Moreover,
from
\[
    \eta_{k+1}^{-1}=Lt_k^2,
    \qquad
    \eta_{k+1}^{-1}=(1-t_k)\eta_k^{-1},
\]
we recover the standard Nesterov/FISTA scalar recursion
\[
    t_k^2=(1-t_k)t_{k-1}^2.
\]
\Cref{thm:generic_accelerated_rate_Mk} gives
\[
    \phi(x_k^{\rm best})-\phi_\star
    \le
    \eta_k^{-1}D_d(x_\star,x_0)
    \le
    \frac{4L}{k^2}D_d(x_\star,x_0)
    =
    \frac{2L}{k^2}\|x_\star-x_0\|^2,
\]
which matches the classical accelerated \(O(L\|x_\star-x_0\|^2/k^2)\) rate up
to the standard indexing convention.

\paragraph{Relatively strongly convex case.}
Let \(\mu>0\) and set
\[
    \alpha_k:=\mu+\eta_k^{-1}.
\]
With \(M_k=L\), the parameters satisfy
\[
    \alpha_{k+1}=Lt_k^2,
    \qquad
    \alpha_{k+1}=(1-t_k)\alpha_k+t_k\mu,
\]
and
\[
    \tau_k
    =
    \frac{t_k-\mu/L}{1-\,\mu/L}.
\]
Define the Euclidean mirror point
\[
    q_{k+1}:=y_k-\frac1\mu \nabla\phi(y_k).
\]
The dual/primal-averaging update can be written as
\[
    z_{k+1}
    =
    \left(1-\frac{t_k\mu}{\alpha_{k+1}}\right)z_k
    +
    \frac{t_k\mu}{\alpha_{k+1}}q_{k+1}.
\]
The specific choice of \(\tau_k\) gives again
\[
    x_{k+1}
    =
    (1-t_k)x_k+t_k z_{k+1}.
\]
Thus the Euclidean strongly convex specialization is the standard Nesterov
estimate-sequence update written in the primal--dual variables. The generic
rate becomes
\[
    \phi(x_k^{\rm best})-\phi_\star
    \le
    \eta_k^{-1}D_d(x_\star,x_0),
    \qquad
    \eta_k^{-1}
    \le
    \eta_K^{-1}\left(1-\sqrt{\frac{\mu}{L}}\right)^{k-K},
\]
once the accepted parameter is \(M_k=L\). This recovers the classical
accelerated linear contraction \(1-\sqrt{\mu/L}\), up to constants and indexing.
If the line search accepts values \(M_k>L\), the updates no longer coincide with Nesterov's, and the
rate is the one \cref{thm:generic_accelerated_rate_Mk} states in terms of the accepted \(M_k\).

\clearpage
\section{Comparison with ABPG and triangle scaling}\label{app:comparison}

\subsection{Comparison with ABPG}

ABPG controls acceleration through a triangle-scaling exponent or gain. The adaptive variants of
\citet{hanzely2021accelerated} search over the exponent or gain, while
\citet{savchuk2024accelerated} fix the exponent and backtrack the relative smoothness constant.
ABrA-GD backtracks both a smoothness parameter and an acceleration parameter. Its search over
\(M_k\) tests the primal--dual bound directly and responds to both smoothness and local
geometric distortion.

Both approaches give bounds in terms of quantities accepted during the run
(\cref{tab:method_comparison}). For ABrA-GD, the same framework also gives a conditional
accelerated linear rate when \(\mu>0\). The cited TSP-based analyses do not establish such a
rate. \citet{hanzely2021accelerated} report faster convergence with restarted ABPG on a
D-optimal design instance described as relatively strongly convex, but give no accelerated linear
guarantee for that experiment.

For gain-adaptive ABPG at \(\gamma=2\), the rate is \(O(\bar G_k/k^2)\), where
\(\bar G_k\) depends on the accepted gains. A uniform bound on the gain can be obtained from
a Hessian ratio over the full region \(\mathcal X\)
\citep[\S2.2, eq.~(16)]{hanzely2021accelerated}. ABrA-GD uses the corresponding local quantity
\(\rho_{d^*}(\Omega_k)\) to bound its dual divergence comparisons. For twice-differentiable
geometries, \cref{prop:dual_bldf_twice_diff} relates these quantities, and
\cref{tab:rho_examples} gives formulas for standard kernels.

\subsection{Relation to the triangle-scaling property}\label{app:tsp}

A uniform triangle-scaling gain at exponent \(2\) bounds the dual norm distortion. Let
\(\mathcal X\) be convex with \(d\) twice differentiable and \(\nabla^2d\succ0\),
let \(\Omega=\nabla d(\mathcal X)\), and suppose the triangle-scaling property holds on
\(\mathcal X\) with exponent \(\gamma=2\) and a gain \(G\) \emph{uniform over triples}. Writing both
divergences in integral form along the common direction, letting \(\theta\to0\) and shrinking the
pair to a point \(q\) gives
\[
G
\;\ge\;
\sup_{x,q\in\mathcal X}
\lambda_{\max}\!\left(\nabla^2d(q)^{-1}\nabla^2d(x)\right)
\;=\;
\rho_{d^*}(\Omega)^2 ,
\]
using \(\nabla^2d^*(\nabla d(x))=\nabla^2d(x)^{-1}\). Thus \(\rho_{d^*}^2\le G\) on the
corresponding dual region. This Hessian ratio appears in the triangle-scaling bound of
\citet[\S2.2, eq.~(16)]{hanzely2021accelerated} and in the Hessian stability condition of
\citet{karimireddy2018global}. Here, it directly controls the local divergence comparisons,
including those needed when \(\mu>0\).

The implication requires a gain uniform over triples. It does not concern the intrinsic
triangle-scaling exponent defined by a pointwise limit and used by adaptive variants in practice.

For a jointly convex Bregman divergence, \(\gamma=1\) is a uniform triangle-scaling exponent at
gain \(1\), giving an \(O(1/k)\) rate. The entropy kernel has this property. For
\(d=\tfrac1p\|x\|^p\) with \(p>2\), the global exponent can also be below \(1\). The kernel-specific
constructions of \citet{hanzely2021accelerated} modify the reference function and restrict the
domain. The formulas in \cref{tab:rho_examples} express the corresponding local distortion through
coordinate ratios: approaching the origin can make the distortion unbounded, while bounded ratios
give finite distortion.

\paragraph{The Burg kernel.}
The Burg kernel \(d(x)=-\sum_i\log x_i\) is not jointly convex, and
\citet{hanzely2021accelerated} show that it has no uniform exponent \(\gamma>1/2\) at gain \(1\).
The global guarantee of ABPG for this kernel is therefore at best \(O(k^{-1/2})\), which is slower
than the \(O(1/k)\) rate of Bregman gradient descent. By \cref{tab:rho_examples}, the dual distortion
factor of the Burg kernel is
\[
\rho_{d^*}(\Omega)=\sup_{x_1,x_2\in\mathcal X}\max_i\frac{(x_2)_i}{(x_1)_i},
\qquad \Omega=\nabla d(\mathcal X),
\]
which is finite whenever the coordinate ratios are bounded on \(\mathcal X\). For example, let
\(\mathcal X_{a,b}=\{x:\ a\le x_i\le b\ \text{for all } i\}\) with \(0<a<b\). Its image
\(\nabla d(\mathcal X_{a,b})\) is a box, hence convex, and \(\rho_{d^*}(\nabla d(\mathcal X_{a,b}))=b/a\).
If, from some iteration on, the points that define the dual segments of \cref{def:omega_k} lie in
\(\mathcal X_{a,b}\), then \(\rho_k\le b/a\). \Cref{thm:eventual_bldf_rate} then gives an
\(O(1/k^2)\) rate, and \cref{cor:iteration_complexity} gives
\(O\big((b/a)\sqrt{L\,D_d(x_\star,x_0)/\varepsilon}\big)\) iterations when \(\mu=0\) and an
accelerated linear rate when \(\mu>0\). For the Burg kernel, the distortion factor thus gives an
accelerated rate where the uniform triangle-scaling exponent gives a rate slower than \(O(1/k)\).

The comparison with a uniform gain at exponent \(2\) is closer. On \(\mathcal X_{a,b}\), the
display above gives \(G\ge(b/a)^2\), so the \(O(G/k^2)\) bound of a uniform gain is not smaller than
the \(O(\rho^2/k^2)\) bound of \cref{thm:eventual_bldf_rate}, up to absolute constants. The
distortion factor \(\rho_k\) is measured on the region of iteration \(k\) only, so it can be smaller
than \(b/a\) when the iterates concentrate. The guarantee of \cref{thm:eventual_bldf_rate} is
conditional: it requires the bound on \(\rho_k\) to hold along the iterates, and the paper does not
prove this bound a priori.
\section{Composite objectives and the subgradient envelope}

\label{app:composite_subgradient}

The implicit envelope viewpoint does not require differentiability of the whole
objective. Consider the composite problem
\[
    \Phi_\star
    =
    \min_{x\in\mathcal X}
    \Phi(x)
    :=
    f(x)+r(x)+\mu D_d(x,x_0),
\]
where \(f\) is differentiable and relatively smooth, while \(r\) is a proper
closed convex term for which a Bregman proximal step is available. For
\(\eta>0\), define
\[
    \Phi^\eta_\star
    :=
    \min_x
    \left\{
    f(x)+r(x)+\left(\mu+\eta^{-1}\right)D_d(x,x_0)
    \right\}.
\]
If an algorithm produces iterates satisfying the requirement
\[
    \Phi(x_k)\le \Phi^\eta_\star,
\]
then the same envelope argument gives
\[
    \Phi(x_k)-\Phi_\star
    \le
    \eta^{-1}D_d(x_\star,x_0),
    \qquad x_\star\in\argmin_x \Phi(x),
\]
because
\[
    \Phi^\eta_\star
    \le
    \Phi^\eta(x_\star)
    =
    \Phi_\star+\eta^{-1}D_d(x_\star,x_0).
\]
Thus the envelope descent requirement survives unchanged in the
composite setting.

A corresponding lower model can be written either with a primal nonsmooth
term or with a subgradient. Dualizing only the smooth part gives
\[
    S_r^\eta(x,\lambda)
    :=
    \langle \lambda,x\rangle-f^*(\lambda)
    +r(x)
    +\left(\mu+\eta^{-1}\right)D_d(x,x_0)
    \le
    \Phi^\eta(x),
\]
and therefore
\[
    \mathcal D_r^\eta(\lambda):=\min_x S_r^\eta(x,\lambda)
    \le
    \Phi^\eta_\star .
\]
Equivalently, for any \(s_y\in\partial r(y)\), convexity of \(r\) gives the
computable subgradient lower model
\[
\begin{aligned}
    \Phi_y^{\rm low}(z)
    :=
    &\;\Phi(y)
    +\left\langle
    \nabla f(y)+s_y+\mu\big(\nabla d(y)-\nabla d(x_0)\big),
    z-y
    \right\rangle
    +\mu D_d(z,y)
    \\
    \le
    &\;\Phi(z).
\end{aligned}
\]
The proximal primal step also remains well defined through
\[
    x^+
    \in
    \argmin_x
    \left\{
    \langle \nabla f(y),x-y\rangle+r(x)+L D_d(x,y)
    +\mu D_d(x,x_0)
    \right\}.
\]
This shows that the envelope and lower-model mechanisms extend naturally to
nonsmooth composite terms. What does not follow automatically is the accelerated
BLDF proof in the main text: the cancellation there uses a single continuous
gradient-generated dual direction, whereas the composite model introduces a
set-valued subgradient or a nonsmooth proximal residual. A full accelerated
composite analysis would therefore require additional control of this residual.

\clearpage
\section{Numerical Experiments}\label{app:numexp}

\subsection{Experimental protocol}
\label{app:numexp_protocol}

The code is the toolbox of \citet{hanzely2021accelerated}\footnote{\url{https://github.com/linxiaolx/accbpg}}
with an implementation of ABrA-GD added as one more method; the feasible domain and its boundary
handling, the datasets, the baseline methods, and the initialization are theirs, unchanged. The
objectives, reference functions and initial points below are stated from that source.

\paragraph{D-optimal design.}
For \(H\in\mathbb R^{m\times n}\) with \(m<n\),
\[
f(x)=-\log\det\big(H\operatorname{diag}(x)H^\top\big),
\qquad
\mathcal X=\Big\{x\ge0:\ \textstyle\sum_{i}x^{(i)}=1\Big\},
\]
and \(d\) is the Burg entropy \(d(x)=-\sum_i\log x^{(i)}\) restricted to that simplex. Here \(f\) is
\(1\)-smooth relative to \(d\) \citep{lu2018relatively}, so \(L=1\). The simplex constraint is the
affine equality of \cref{app:affine_restriction}. The restriction makes \(\dom d\) bounded, so
\(\dom d^*=\mathbb R^n\) and the virtual point \(q_{k+1}\) of
\cref{lem:bound_acc_strongly_convex} exists at every iterate on these instances. Every method starts
at the simplex center \(x_0=(1/n)\mathbf 1\). The instances are built from the LIBSVM regression
datasets abalone (\(n=4177\), \(m=8\)), bodyfat (\(252\), \(14\)), mpg (\(392\), \(7\)) and housing
(\(506\), \(13\)) \citep{chang2011libsvm}.

The relatively strongly convex instances are formed by the anchor term already in
\cref{eq:phi_objective_boxed}: \(\phi(x)=f(x)+\mu D_d(x,x_0)\), with the same \(f\), \(d\), and
\(x_0\) as the convex instance and \(\mu=10^{-4}\).

\paragraph{Poisson linear inverse problem.}
\(f(x)=D_{\mathrm{KL}}(b,Ax)\) on \(\mathcal X=\mathbb R^n_+\), with \(A\in\mathbb R^{m\times n}\)
drawn entrywise uniform on \([0,1]\) and column-normalized, and \(b=A\bar x+\text{noise}\) for a
sparse nonnegative \(\bar x\). The reference function is again the Burg entropy, now on the orthant,
and \(f\) is \(\|b\|_1\)-smooth relative to it \citep{bauschke2017descent}. 

\paragraph{Parameters.}
ABrA-GD starts every run at \(L_0=1\) and \(M_{-1}=L_0\) on both problems. This differs from
\citet{hanzely2021accelerated}, who supply the theoretical constant: \(L=1\) for D-optimal design and
\(L=\|b\|_1\) for the Poisson problem. The baselines are run through their own interface with the
constants that interface expects, as in \citet{hanzely2021accelerated}: BPG and ABPG receive \(L\),
and BPG-LS, ABPG-e and ABPG-g adjust a gain online starting from it. No relative strong convexity
constant is supplied to any baseline, because that interface takes no \(\mu\); on the \(\mu>0\)
instances the anchor term enters through \(\phi\) itself.

\subsection{Safeguards and ablations}

\paragraph{Restart strategies.}
As in adaptive ABPG \citep{hanzely2021accelerated}, we test two standard
restart heuristics as ablations. The first is a function-value restart,
triggered when the candidate objective value increases. The second is a
gradient restart, triggered when the accelerated direction becomes misaligned
with the current descent direction. After a restart, we keep the current best
primal point and reset the accelerated state by setting
\[
    z_k=x_k=x_k^{\mathrm{best}},
\]
resetting the dual averaging state, and continuing with the same adaptive
searches for \(L_{\mathrm{cur}}\) and \(M_k\).

These restart rules are heuristic and are not part of ABrA-GD's convergence
analysis. They are included only to test whether explicit momentum resets add
anything beyond the exit condition of \cref{algo:bregman_adaptive}. The
experiments below show that restarts do not provide a consistent advantage; in
several cases they increase the variability of the accepted acceleration
parameter \(M_k\) or reduce oracle-call efficiency.

\paragraph{Gradient-descent safeguard.}
The accelerated step may become inefficient when the accepted acceleration
parameter \(M_k\) is very large. To make the method never worse than Bregman
gradient descent, one can compare the certificate progress of the accelerated
step with the nonaccelerated dual-averaging update. This safeguard uses the true
relative smoothness constant \(L\) only as a theoretical comparison.
Define
\[
t_k^{\rm gd}
=
\frac{\mu+\eta_k^{-1}}{L+\eta_k^{-1}},
\qquad
\bar\eta_{k+1}^{-1}
=
(1-t_k^{\rm gd})\eta_k^{-1}.
\]
After the accelerated line search has produced a candidate
\[
    \eta_{k+1}^{-1}=M_k t_k^2-\mu,
\]
we accept it only if
\[
    \eta_{k+1}^{-1}\le \bar\eta_{k+1}^{-1}.
\]
Otherwise, we perform the certified nonaccelerated primal--dual step obtained
by setting
\[
    y_k=z_k,
    \qquad
    t_k=t_k^{\rm gd},
    \qquad
    \eta_{k+1}^{-1}=\bar\eta_{k+1}^{-1}.
\]
This fallback is exactly a Bregman gradient step in primal variables, but
written in the primal--dual variables so that the envelope certificate is
preserved. Consequently,
\[
\phi(x_k^{\rm best})-\phi^\star
\le
\eta_{k,\rm gd}^{-1}D_d(x^\star,x_0),
\]
so the safeguarded method has at least the convergence rate of Bregman gradient
descent while retaining accelerated steps whenever their certificate progress
is better.

\textbf{In all reported experiments, the safeguard was never triggered}: every accepted
ABrA-GD step already produced at least as much certificate progress as the
nonaccelerated update. Thus, the safeguard is included only to make the
worst-case comparison with Bregman gradient descent explicit and does not affect
the numerical results.

\subsection{Additional numerical experiments}

We report additional experiments on D-optimal design and Poisson linear inverse
problems. For D-optimal design, we test both the convex case and a relatively
strongly convex variant with \(\mu=10^{-4}\). For the Poisson inverse problem,
we test two variants, denoted Poisson L1 and Poisson L2.

Across the convex D-optimal design instances, ABrA-GD improves on nonaccelerated
Bregman gradient methods and matches the best ABPG variants over the plotted
range. In the relatively strongly convex setting, ABrA-GD's
log-residual decreases approximately linearly over the plotted range and
reaches a lower residual than the competing methods. On the Poisson problems,
ABrA-GD again improves on the nonaccelerated baselines, but which method
reaches the lowest residual depends more strongly on the instance and the
variant. The restart comparisons below show that explicit
restarts do not improve on the plain method in these instances.

\clearpage

\subsubsection{D-optimal design, convex, $\gamma=1$ for ABPG}

\begin{figure}[h]
    \centering\includegraphics[width=1\linewidth]{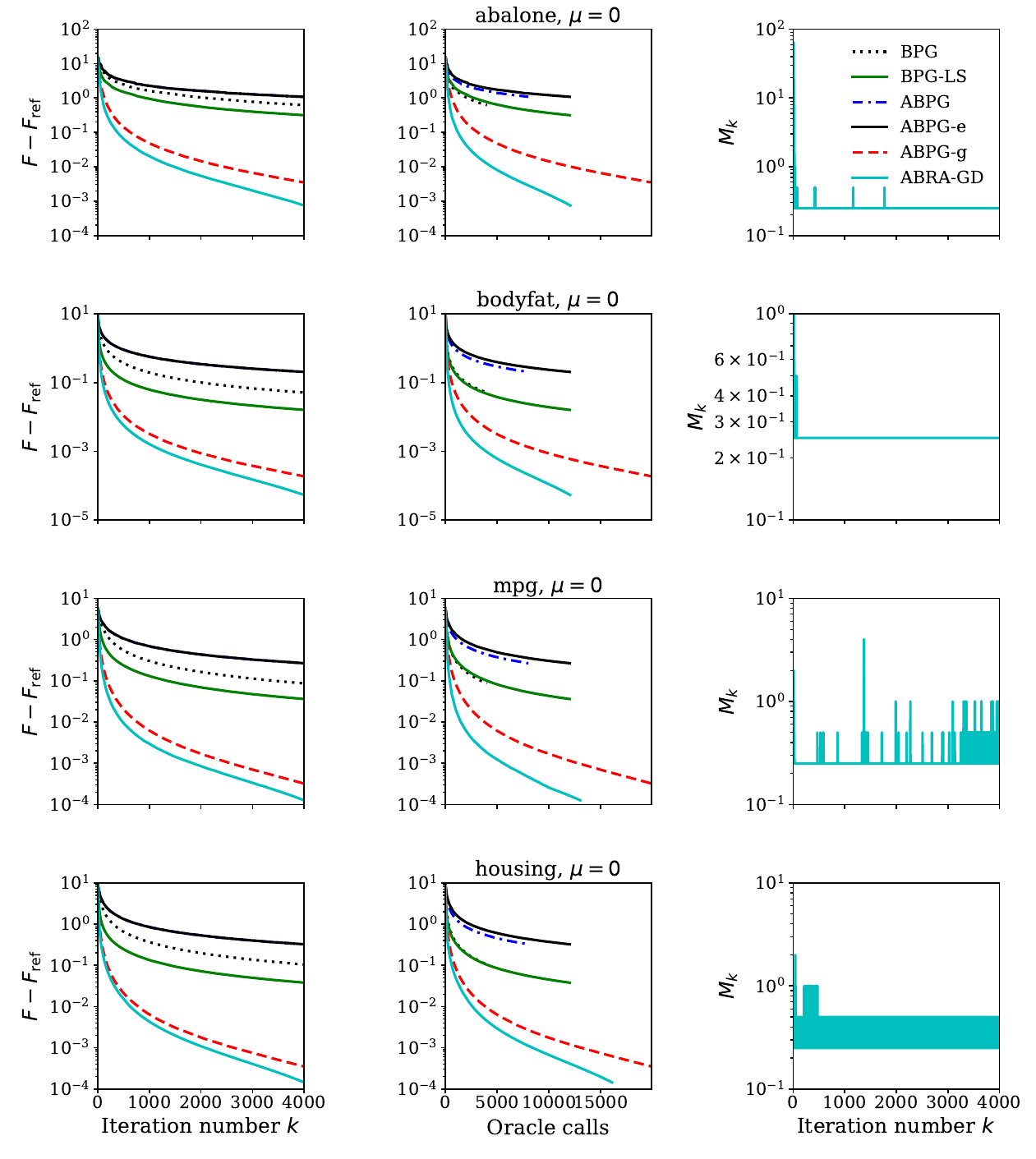}
    \caption{
    D-optimal design on the housing dataset with ABPG run at the
    triangle-scaling exponent \(\gamma=1\), one point in
    \citet{hanzely2021accelerated}'s own sweep over
    \(\gamma\in\{1.0,1.5,2.0,2.2\}\); they report that at \(\gamma=1\) ABPG
    converges at \(O(k^{-1})\) but is slower than BPG. They fix \(\gamma=2\)
    for every ABPG variant in their own comparisons, which is why we use
    \(\gamma=2\) in all other comparisons.
    }
    \label{fig:dopt_convex_bundle_gamma1}
\end{figure}

\clearpage
\subsubsection{D-optimal design, convex}

\begin{figure}[h]
    \centering
    \includegraphics[width=1\linewidth]{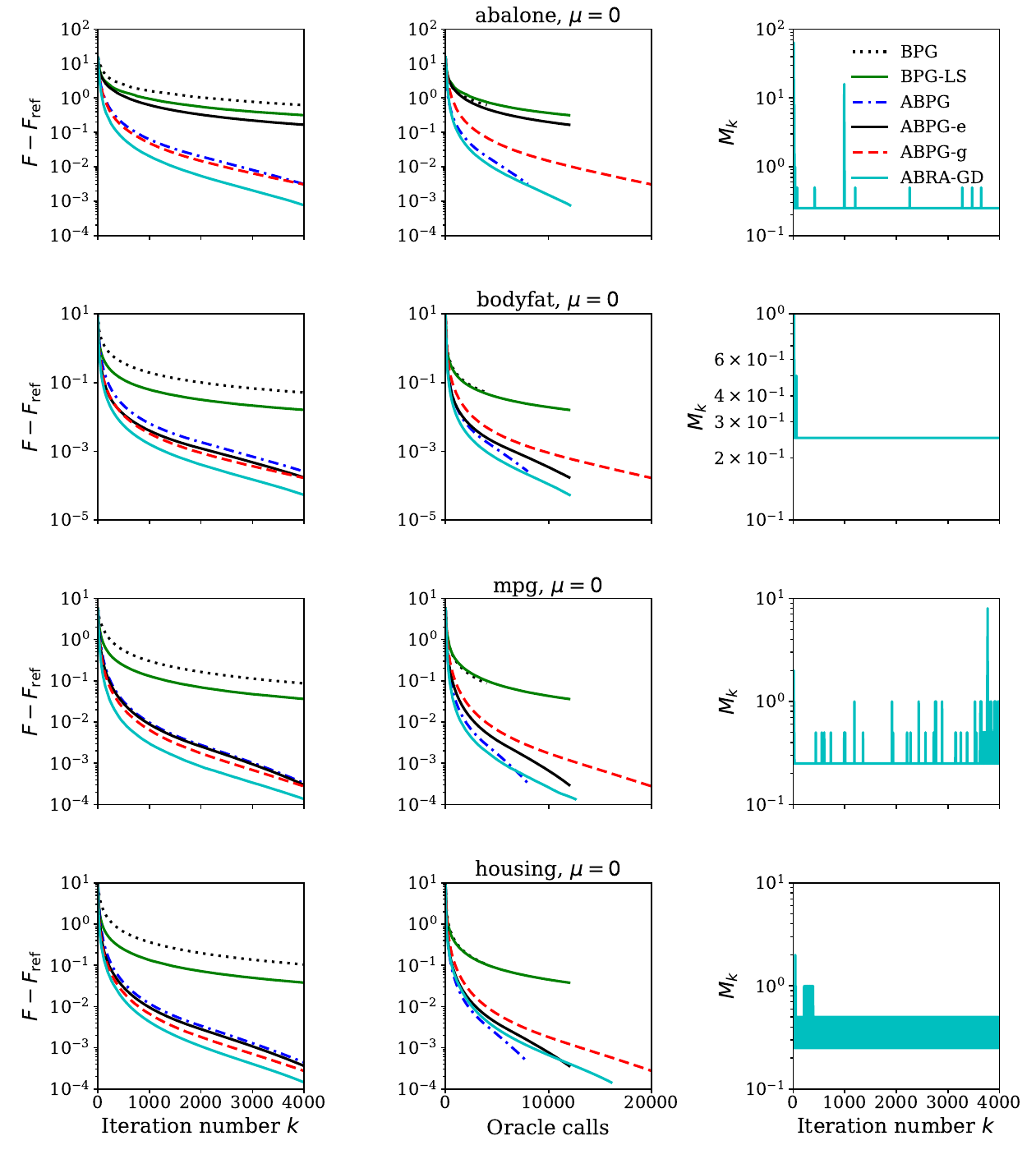}
    \caption{
    Convex D-optimal design experiments on the abalone, bodyfat, mpg, and
    housing datasets. For each dataset, we report the residual
    \(F-F_\star\) versus iterations, the same residual versus oracle calls,
    and the accepted values of \(M_k\) for ABrA-GD. ABrA-GD improves
    over BPG and BPG-LS on all datasets and matches the best accelerated ABPG
    variants over the plotted range, while adapting \(M_k\) locally during the
    run.
    }
    \label{fig:dopt_convex_bundle}
\end{figure}

\clearpage
\subsubsection{D-optimal design, strongly convex}

\begin{figure}[h]
    \centering
    \includegraphics[width=1\linewidth]{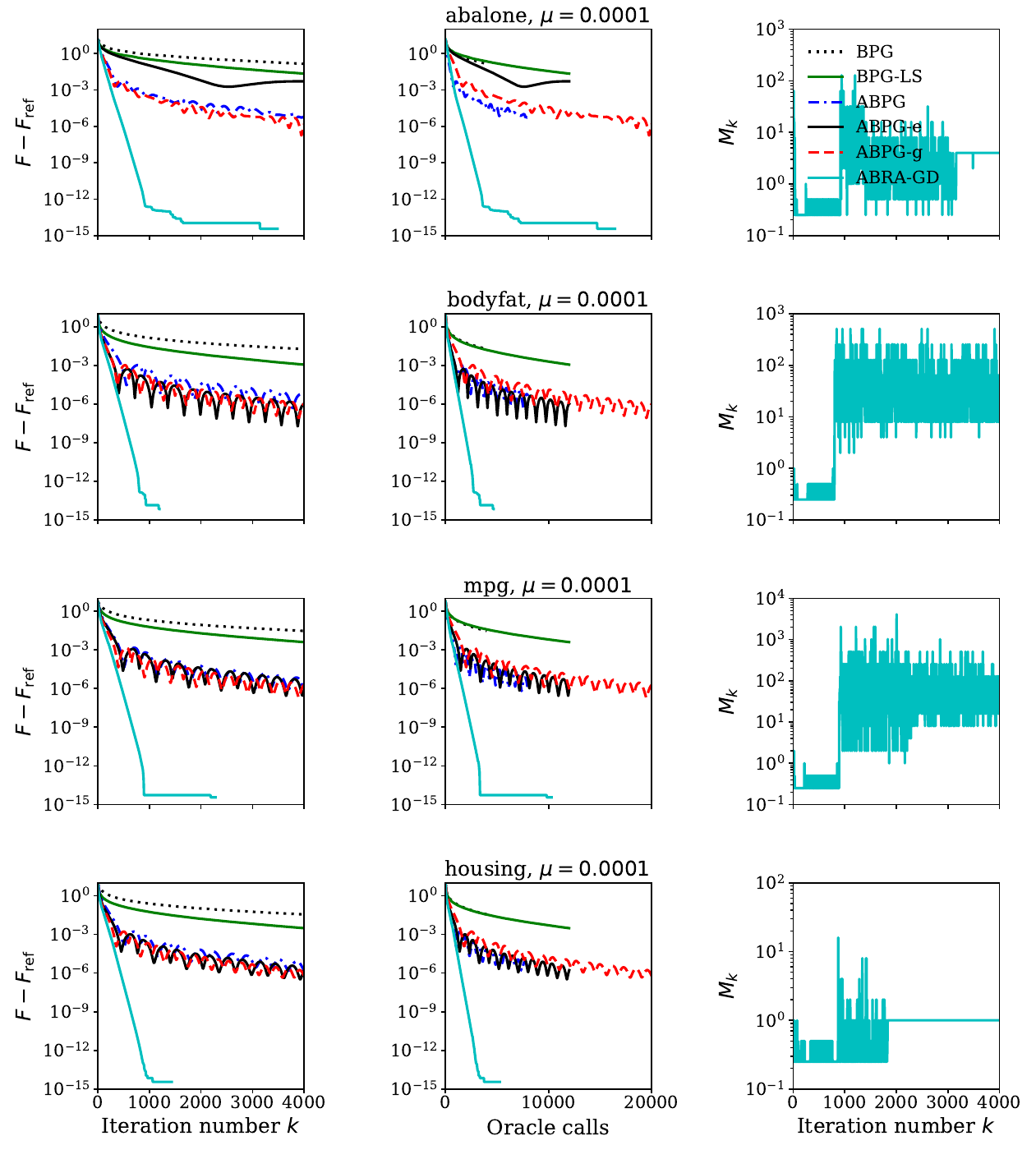}
    \caption{
    Relatively strongly convex D-optimal design experiments with
    \(\mu=10^{-4}\).
    The left and middle columns show convergence in iterations
    and oracle calls, while the right column reports the accepted values of
    \(M_k\). In this regime, ABrA-GD's log-residual decreases approximately
    linearly over the plotted range and reaches a lower residual than BPG,
    BPG-LS, and the ABPG variants. The accepted values of \(M_k\) can
    fluctuate; \(\phi^{\rm best}_k\) is non-increasing by construction
    (\cref{algo:bregman_adaptive}) regardless.
    }
    \label{fig:dopt_strcvx_bundle}
\end{figure}

\clearpage
\subsubsection{Poisson problem}

\begin{figure}[h]
    \centering
    \includegraphics[width=1\linewidth]{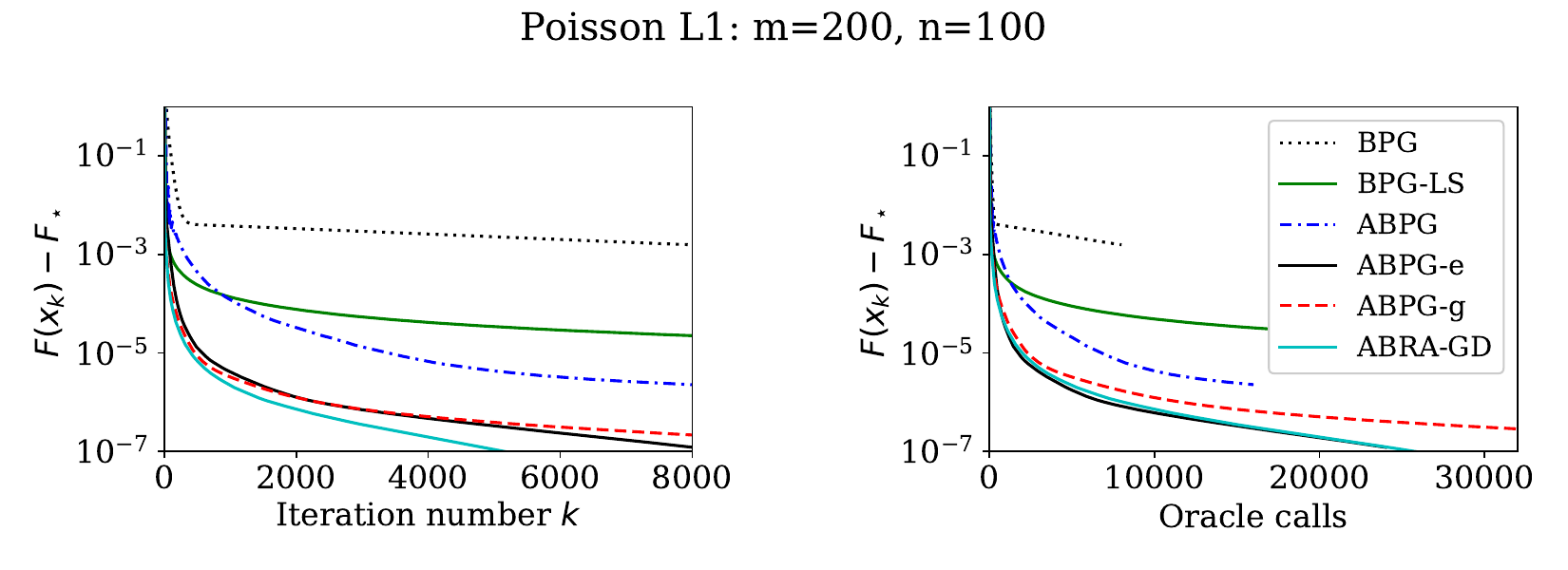}
    \caption{
    Poisson linear inverse problem, L1 variant
    \((m=200,n=100)\). ABrA-GD improves on the nonaccelerated Bregman methods
    and matches the best ABPG variants over the plotted range, in both
    iteration count and oracle-call count.
    }
    \label{fig:poisson_l1}
\end{figure}

\begin{figure}[h]
    \centering
    \includegraphics[width=1\linewidth]{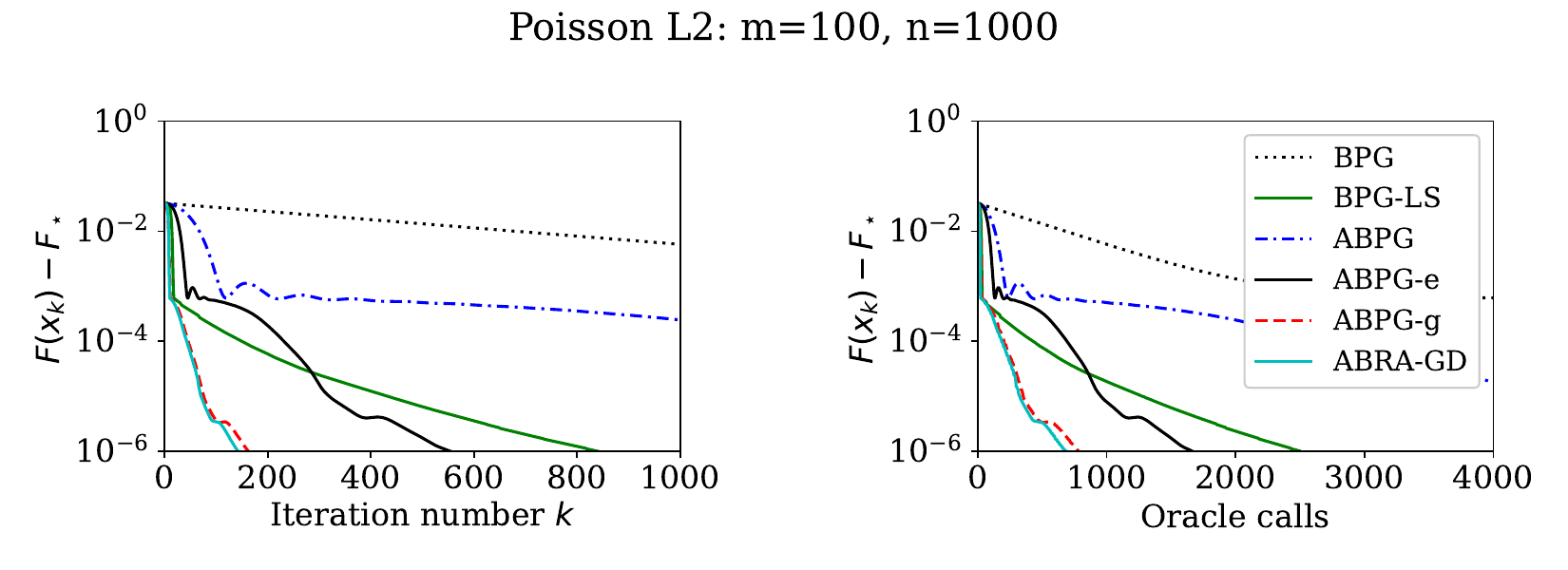}
    \caption{
    Poisson linear inverse problem, L2 variant
    \((m=100,n=1000)\). The accelerated methods reach a lower residual than
    BPG and BPG-LS. ABrA-GD tracks the best accelerated baseline over the
    plotted range; which method reaches the lowest final residual is more
    sensitive to the instance and the adaptive strategy than
    in the D-optimal design experiments.
    }
    \label{fig:poisson_l2}
\end{figure}

\clearpage
\subsubsection{Restart comparison: D-optimal design, convex}

\begin{figure}[h]
    \centering
    \includegraphics[width=\linewidth]{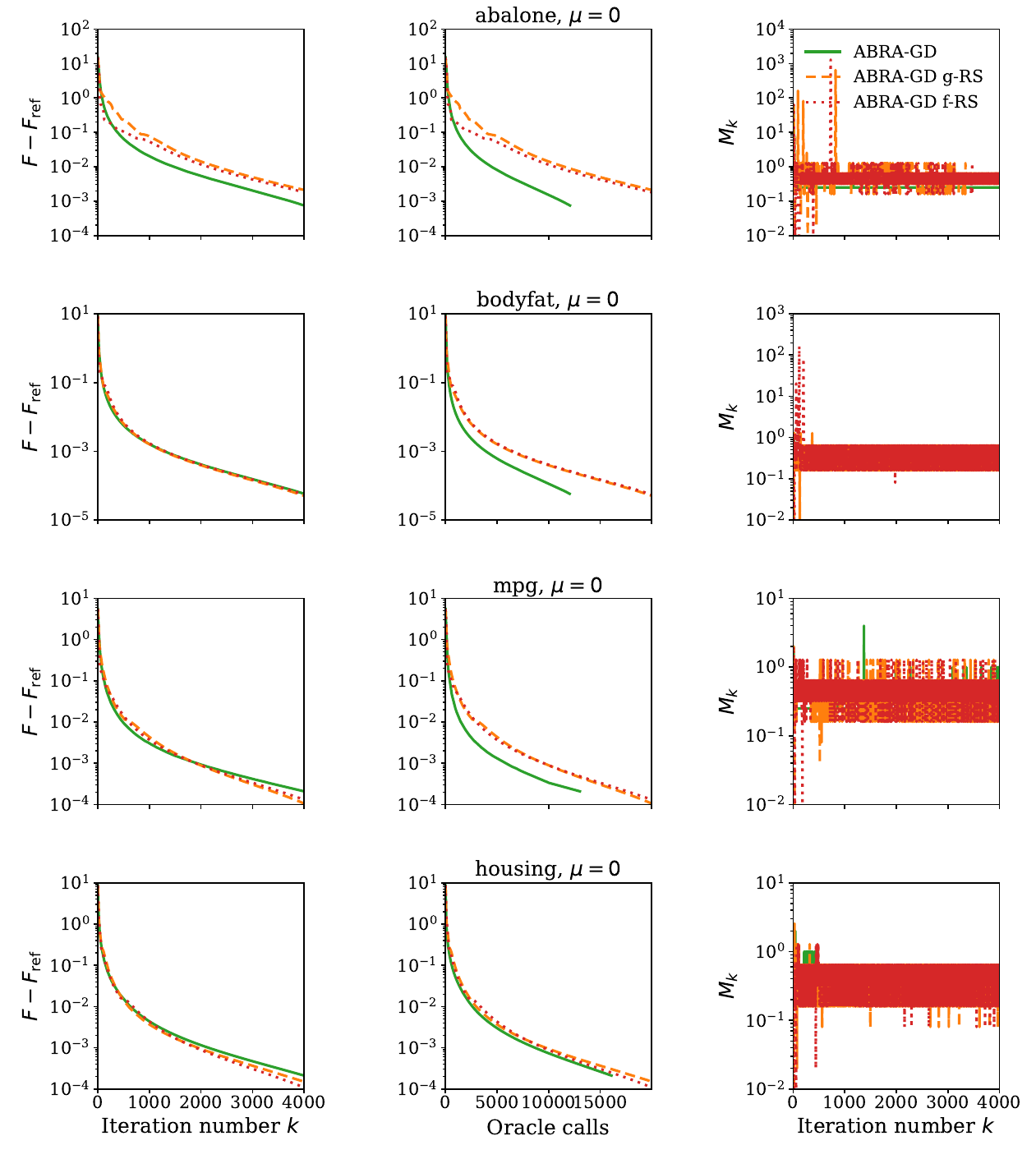}
    \caption{
    Restart comparison for convex D-optimal design. We compare ABrA-GD without
    restart, with gradient restart (g-RS), and with function-value restart
    (f-RS). Restarts do not improve performance consistently: the plain ABrA-GD
    variant is usually comparable or better in oracle calls, while the restarted
    variants can introduce additional variability in \(M_k\).
    }
    \label{fig:dopt_convex_restart}
\end{figure}

\clearpage
\subsubsection{Restart comparison: D-optimal design, strongly convex}

\begin{figure}[h]
    \centering
    \includegraphics[width=\linewidth]{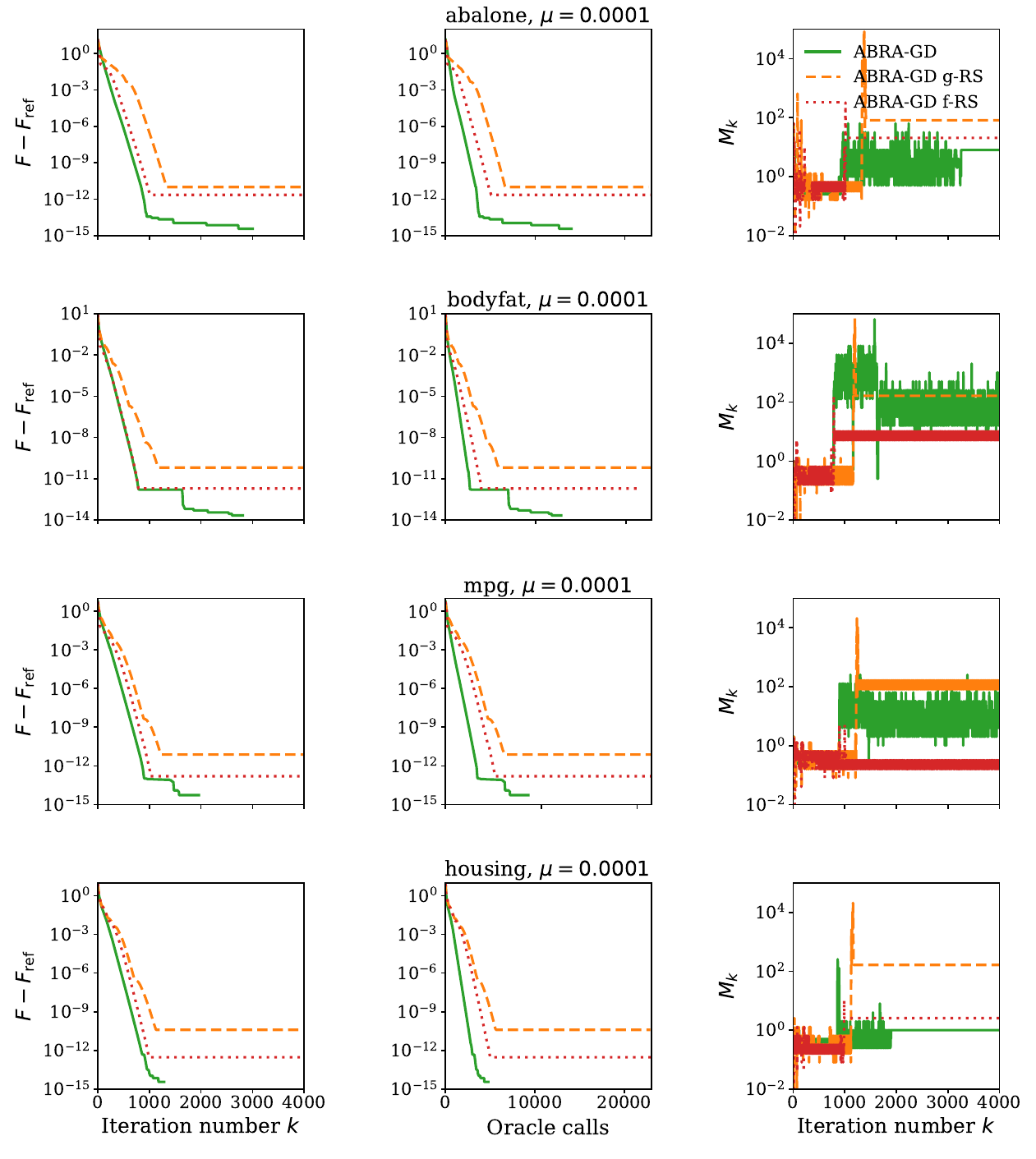}
    \caption{
    Restart comparison for relatively strongly convex D-optimal design with
    \(\mu=10^{-4}\). The non-restarted ABrA-GD variant typically reaches the
    lowest final residual across the datasets. The restarted variants often
    plateau earlier or require larger accepted values of \(M_k\).
    }
    \label{fig:dopt_strcvx_restart}
\end{figure}

\clearpage
\subsubsection{Restart comparison: Poisson problem}

\begin{figure}[h]
    \centering
    \includegraphics[width=\linewidth]{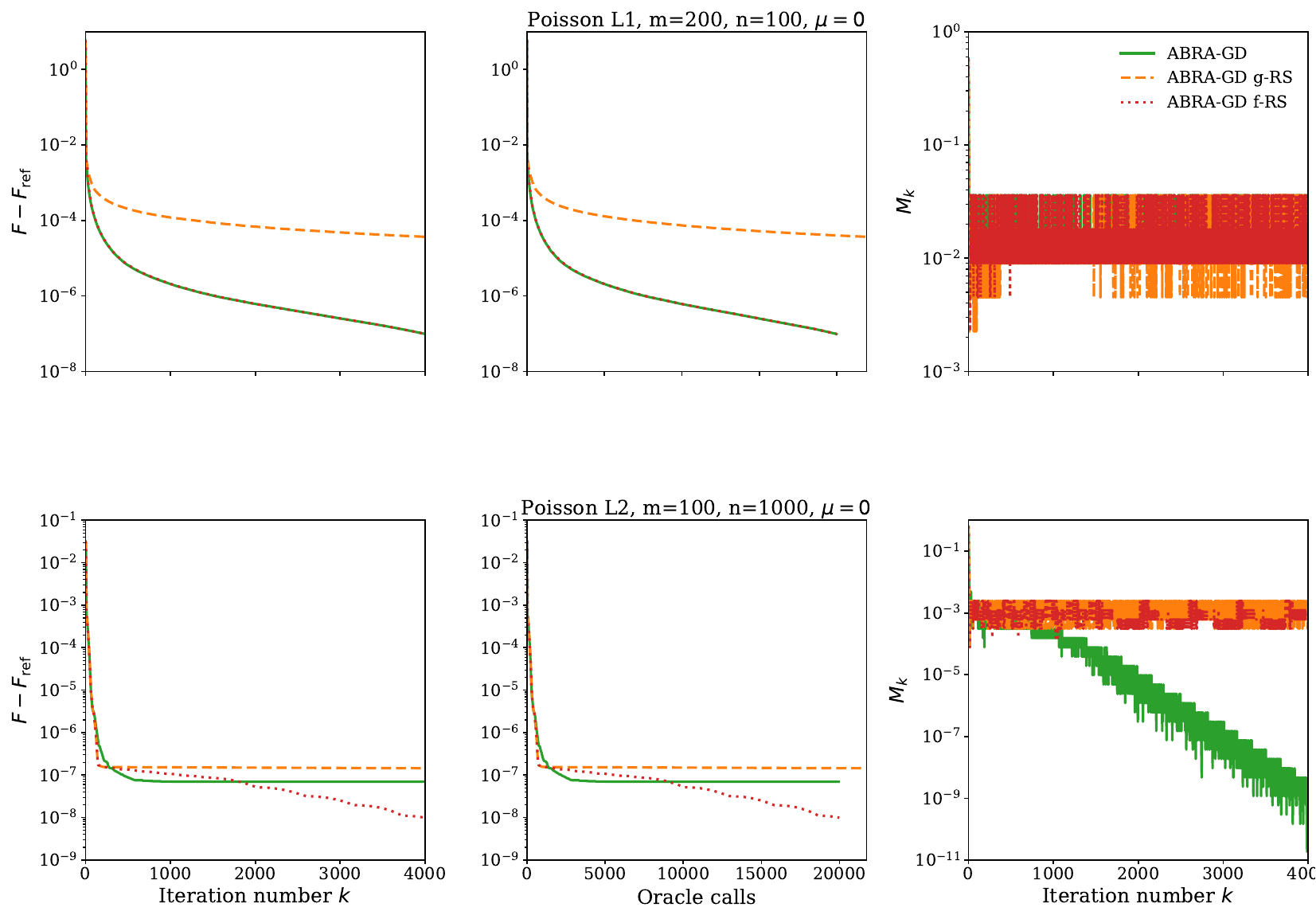}
    \caption{
    Restart comparison for the Poisson linear inverse problems. Restart is
    neutral or worse on both instances: the plain ABrA-GD update is not
    improved upon by either restarted variant on either instance. The
    restarted variants both stall on the Poisson L2 instance.
    }
    \label{fig:poisson_restart}
\end{figure}

\clearpage

\end{document}